\documentclass[12pt]{amsart}
\usepackage{}

\usepackage{amsmath}
\usepackage{amsfonts}
\usepackage{amssymb}
\usepackage[all]{xy}           

\usepackage{bbding}
\usepackage{txfonts}
\usepackage{amscd}

\usepackage[shortlabels]{enumitem}
\usepackage{ifpdf}
\ifpdf
  \usepackage[colorlinks,final,backref=page,hyperindex]{hyperref}
\else
  \usepackage[colorlinks,final,backref=page,hyperindex,hypertex]{hyperref}
\fi
\usepackage{tikz}
\usepackage[active]{srcltx}

\makeatletter

\newtheorem{thm}{Theorem}[section]
\newtheorem{lem}[thm]{Lemma}
\newtheorem{cor}[thm]{Corollary}
\newtheorem{pro}[thm]{Proposition}
\newtheorem{ex}[thm]{Example}
\newtheorem{rmk}[thm]{Remark}
\newtheorem{defi}[thm]{Definition}

\newcommand {\emptycomment}[1]{}
\newcommand {\yh}[1]{{\marginpar{*}\scriptsize\textcolor{purple}{yh: #1}}}

\newcommand{\lon }{\,\rightarrow\,}
\newcommand{\be }{\begin{equation}}
\newcommand{\ee }{\end{equation}}

\newcommand{\g}{\mathfrak g}

\newcommand{\huaB}{\mathcal{B}}

\newcommand{\huaL}{\mathcal{L}}

\newcommand{\huaD}{\mathcal{D}}

\newcommand{\huaO}{{\mathcal{O}}}

\newcommand{\CWM}{C^{\infty}(M)}

\newcommand{\frkB}{\mathfrak B}

\newcommand{\frkL}{\mathfrak L}

\newcommand{\Id}{\rm{Id}}

\newcommand{\br}[1]{   [ \cdot,    \cdot  ]   }

\newcommand{\dM}{\mathrm{d}}

\newcommand{\Hom}{\mathrm{Hom}}

\newcommand{\gl}{\mathfrak {gl}}

\newcommand{\ad}{\mathrm{ad}}

\newcommand{\K}{\mathbb{K}}

\newcommand{\reg}{\mathrm{reg}}

\newcommand{\PreLie}{\mathsf{PreLie}}
\newcommand{\Zinb}{\mathsf{Zinb}}
\newcommand{\Perm}{\mathsf{Perm}}
\newcommand{\Leib}{\mathsf{Leib}}
\newcommand{\FMan}{\mathsf{FMan}}
\newcommand{\PreFMan}{\mathsf{PreFMan}}

\begin{document}

\title[$F$-manifold algebras and deformation quantization via pre-Lie algebras]{$F$-manifold algebras and deformation quantization via pre-Lie algebras}

\author{Jiefeng Liu}
\address{School of Mathematics and Statistics, Northeast Normal University, Changchun 130024, Jilin, China}
\email{liujf12@126.com}

\author{Yunhe Sheng}
\address{Department of Mathematics, Jilin University, Changchun 130012, Jilin, China}
\email{shengyh@jlu.edu.cn}

\author{Chengming Bai}
\address{Chern Institute of Mathematics and LPMC, Nankai University,
Tianjin 300071, China}
\email{baicm@nankai.edu.cn}
\vspace{-5mm}


\begin{abstract}
The notion of an $F$-manifold algebra is the underlying algebraic structure of an $F$-manifold. We introduce the notion of  pre-Lie formal deformations of commutative associative algebras and show that $F$-manifold algebras are the corresponding semi-classical limits. We study   pre-Lie infinitesimal deformations and extension of pre-Lie $n$-deformation to pre-Lie $(n+1)$-deformation of a commutative associative algebra through the cohomology groups of pre-Lie algebras. We introduce the notions of   pre-$F$-manifold algebras and dual pre-$F$-manifold algebras, and show that a pre-$F$-manifold algebra gives
rise to an $F$-manifold algebra through the sub-adjacent associative algebra and the sub-adjacent Lie algebra.  We use   Rota-Baxter operators, more generally $\huaO$-operators and average operators on $F$-manifold algebras   to construct pre-$F$-manifold algebras and dual pre-$F$-manifold algebras.
\end{abstract}




\maketitle

\tableofcontents

\allowdisplaybreaks


\section{Introduction}\label{sec:intr}

The concept of Frobenius manifolds was introduced by Dubrovin in \cite{Dub95} as a geometrical manifestation of the Witten-Dijkgraaf-Verlinde-Verlinde (WDVV) associativity equations in the $2$-dimensional topological field theories. Hertling and Manin weakened the conditions of a Frobenius manifold and introduced the notion of an $F$-manifold in \cite{HerMa}. Any Frobenius manifold is an $F$-manifold. $F$-manifolds   appear in many fields of mathematics such as singularity theory \cite{Her02}, quantum K-theory \cite{LYP}, integrable systems \cite{DS04,DS11,LPR11},  operad \cite{Merku} and so on.

Motivated by the study of the operad of the underlying algebraic structure of an $F$-manifold,  Dotsenko introduced the notion of an $F$-manifold algebra in \cite{Dot} and showed that the graded object of the filtration of the operad encoding pre-Lie algebras is the operad encoding $F$-manifold algebras. The slogan for $F$-manifold algebras promoted by Dotsenko is that ``$F$-manifold algebras are the same to pre-Lie algebras as Poisson algebras to associative algebras". It is well-known that Poisson algebras can be understood as semi-classical limits of associative formal deformations of commutative associative algebras. So it is natural to ask whether $F$-manifold algebras can be understood as certain semi-classical limits.

The first aim of this paper is to answer the above question. We introduce the notion of  pre-Lie formal deformations of commutative associative algebras and show that $F$-manifold algebras are the corresponding semi-classical limits. Note that a commutative pre-Lie algebra is associative and hence the aforementioned formal deformation can put into
an extent of ``pure" pre-Lie algebras, that is, $F$-manifold algebras are the semi-classical limits of pre-Lie formal deformations of commutative pre-Lie algebras. This result is parallel to that the semi-classical limit of an associative formal deformation of a commutative associative algebra is a Poisson algebra and   illustrates the slogan promoted by Dotsenko. Viewing the commutative associative algebra as a pre-Lie algebra, we show that pre-Lie infinitesimal deformation and extension of pre-Lie $n$-deformation to pre-Lie $(n+1)$-deformation of a commutative associative algebra are classified by the second and the third cohomology groups of the pre-Lie algebra respectively.

The notion of a pre-Poisson algebra was introduced by Aguiar in \cite{A2}, which combines a Zinbiel algebra and a pre-Lie algebra such that some compatibility conditions are satisfied. Zinbiel algebras (also called dual Leibniz algebras) were introduced by Loday in \cite{Lod1} in his study of the algebraic structure behind the cup product on the cohomology groups of the Leibniz algebra. See \cite{Liv,Lod2} for more details on Zinbiel algebras. More importantly, a pre-Poisson algebra gives rise to a  Poisson algebra naturally through the  sub-adjacent commutative associative algebra of the Zinbiel algebra and the sub-adjacent  Lie algebra of the pre-Lie algebra. Conversely,  a Rota-Baxter operator action (more generally an $\huaO$-operator action) on a Poisson algebra gives rise to a pre-Poisson algebra. In this paper, we introduce the notion of a pre-$F$-manifold algebra, which also contains a Zinbiel algebra and a pre-Lie algebra, such that some compatibility conditions are satisfied.  By the sub-adjacent associative algebra and the sub-adjacent Lie algebra, a pre-$F$-manifold algebra  gives rise to an $F$-manifold algebra naturally. We further introduce the notion of a Rota-Baxter operator (more generally an $\huaO$-operator) on an $F$-manifold algebra, which is simultaneously a Rota-Baxter operator on the underlying associative algebra and a Rota-Baxter operator on the underlying Lie algebra. See \cite{A2,Bai2007,BGN2013A,CK,EGK,Gub,GK} for more details on Rota-Baxter operators and $\huaO$-operators. A pre-$F$-manifold algebra can be obtained through the action of a Rota-Baxter operator (more generally an $\huaO$-operator). This coincides with the general theory of splitting of operads \cite{BBGN,PBG}.
 The above relations can be summarized into the following
commutative diagram:
$$
\xymatrix{
\ar[rr] \mbox{{ Zinbiel algebra} + pre-Lie algebra }\ar[d]_{\mbox{sub-adjacent~~}}
                && \mbox{pre-$F$-manifold algebra}\ar[d]_{\mbox{sub-adjacent~~}}\\
\ar[rr] \mbox{{ comm associative algebra} + Lie algebra }\ar@<-1ex>[u]_{\mbox{~~Rota-Baxter ~action}}
                && \mbox{ $F$-manifold algebra. }\ar@<-1ex>[u]_{\mbox{~~Rota-Baxter~ action}}}
$$

It is well known that the Koszul duals of
the operad $\PreLie$ of pre-Lie algebras
and the operad $\Zinb$ of Zinbiel algebras
are the operad $\Perm$ of permutative
algebras and the operad $\Leib$ of Leibniz
algebras respectively (\cite{Lod2,ChaLiv}). We further introduce
the notion of a dual pre-$F$-manifold algebra, which contains a permutative
algebra and a Leibniz algebra such that some compatibility
conditions hold. By the action of an average operator on an
$F$-manifold algebra, we can obtain a dual pre-$F$-manifold algebra.

The paper is organized as follows. In Section \ref{sec:Preliminaries}, we recall pre-Lie algebras, Lie-admissible algebras and the cohomology theory of pre-Lie algebras. In Section \ref{sec:F-algebras and deformations}, first we introduce the notion of $F$-manifold-admissible algebras, which gives rise to $F$-manifold algebras. Various examples on $F$-manifold algebras are given.  Then we introduce the notion of  pre-Lie formal deformations of commutative associative algebras and show that $F$-manifold algebras are the corresponding semi-classical limits. Furthermore, we study pre-Lie infinitesimal deformations and extensions of pre-Lie $n$-deformations to pre-Lie $(n+1)$-deformations of a commutative associative algebra. In Section \ref{sec:pre-F-algebras},  first we study representations of an $F$-manifold algebra. Then we introduce the notions of pre-$F$-manifold algebras and Rota-Baxter operators (more generally $\huaO$-operators) on an $F$-manifold algebra. We show that on one hand, an $\huaO$-operator on an $F$-manifold algebra gives a pre-$F$-manifold algebra, and on the other hand, a pre-$F$-manifold algebra naturally gives an $\huaO$-operator on the sub-adjacent $F$-manifold algebra. More examples of pre-$F$-manifold algebras and $F$-manifold algebras are given. In Section \ref{sec:dual pre-F-algebras}, we introduce the notions of dual pre-$F$-manifold algebras and average operators on an $F$-manifold algebra and show that an average operator on an $F$-manifold algebra gives a dual pre-$F$-manifold algebra naturally.

In this paper, all the vector spaces are over algebraically closed field $\mathbb K$ of characteristic $0$, and finite dimensional.

\vspace{2mm}

\noindent
{\bf Acknowledgements.} This research was  supported by NSFC (11922110,11901501,11931009). C. Bai is also supported by the Fundamental Research Funds for the Central Universities and Nankai ZhiDe Foundation.

\section{Preliminaries}\label{sec:Preliminaries}
In this section, we briefly recall pre-Lie algebras, Lie-admissible algebras and the cohomology theory of pre-Lie algebras.

Pre-Lie algebras are a class of nonassociative algebras appearing
in many fields in
mathematics and mathematical physics (cf. \cite{Bakalov,Ban,ChaLiv,DSV,Lichnerowicz,MT}, and the survey \cite{Burde}). In particular,
as pointed out in \cite{Dot}, the operad of pre-Lie algebras is ``one of the most famous operads in the literature".

\begin{defi}  A {\bf pre-Lie algebra} is a pair $(\g,\ast_\g)$, where $\g$ is a vector space and  $\ast_\g:\g\otimes \g\longrightarrow \g$ is a bilinear multiplication satisfying that for all $x,y,z\in \g$, the associator
$(x,y,z)=(x\ast_\g y)\ast_\g z-x\ast_\g(y\ast_\g z)$ is symmetric in $x,y$,
i.e.
$$(x,y,z)=(y,x,z),\;\;{\rm or}\;\;{\rm
equivalently,}\;\;(x\ast_\g y)\ast_\g z-x\ast_\g(y\ast_\g z)=(y\ast_\g x)\ast_\g z-y\ast_\g(x\ast_\g z).$$
\end{defi}

It is obvious that any associative algebra is a  pre-Lie algebra.

\begin{lem}{\rm (\cite{Burde})}
A commutative pre-Lie algebra is associative.
\end{lem}

\begin{lem}{\rm(\cite{Burde})}\label{lem:pre-Lie-Lie} Let $(\g,\ast_\g)$ be a pre-Lie algebra. The commutator
$ [x,y]_\g=x\ast_\g y-y\ast_\g x$ defines a Lie algebra structure on $\g$,
which is called the {\bf sub-adjacent Lie algebra} of $(\g,\ast_\g)$ and denoted by $\g^c$. Furthermore, $L:\g\rightarrow
\gl(\g)$ defined by
\begin{equation}\label{eq:defiLpreLie}
L_xy=x\ast_\g y,\quad \forall x,y\in \g
\end{equation}
 gives a representation of $\g^c$ on $\g$.
\end{lem}

A {\bf Lie-admissible algebra} is a nonassociative algebra $(\g,\ast_\g)$ whose commutator algebra is a Lie algebra.
More precisely, it is equivalent to the following condition:
\begin{equation}
  (x,y,z)-(y,x,z)+(y,z,x)-(z,y,x)+(z,x,y)-(x,z,y)=0,\quad \forall x,y,z\in \frak g.
\end{equation}

Obviously, a pre-Lie algebra is a Lie-admissible algebra.

\begin{defi}
Let $(\g,\ast_\g)$ be a pre-Lie algebra and $V$  a vector
space. A {\bf representation} of $\g$ on $V$ consists of a pair
$(\rho,\mu)$, where $\rho:\g\longrightarrow \gl(V)$ is a representation
of the Lie algebra $\g^c$ on $V $ and $\mu:\g\longrightarrow \gl(V)$ is a linear
map satisfying \begin{eqnarray}\label{representation condition 2}
 \rho(x)\mu(y)u-\mu(y)\rho(x)u=\mu(x\ast_\g y)u-\mu(y)\mu(x)u, \quad \forall~x,y\in \g,~ u\in V.
\end{eqnarray}
\end{defi}

Usually, we denote a representation by $(V;\rho,\mu)$. Let $R:\g\rightarrow
\gl(\g)$ be a linear map with $x\longrightarrow R_x$, where the
linear map $R_x:\g\longrightarrow\g$  is defined by
$R_x(y)=y\ast_\g x,$ for all $x, y\in \g$. Then
$(\g;\rho=L,\mu=R)$ is a representation, which we call the
{\bf regular representation}. Define two linear maps $L^*,R^*:\g\longrightarrow
\gl(\g^*)$   with $x\longrightarrow L^*_x$ and
$x\longrightarrow R^*_x$ respectively (for all $x\in \g$)
by
\begin{equation}
\langle L_x^*(\xi),y\rangle=-\langle \xi, x\ast_\g y\rangle, \;\;
\langle R_x^*(\xi),y\rangle=-\langle \xi, y\ast_\g x\rangle, \;\;
\forall x, y\in \g, \xi\in \g^*.
\end{equation}
Then $(\g^*;\rho={\rm ad}^*=L^*-R^*, \mu=-R^*)$ is a
representation of $(\g,\cdot_\g)$.

The cohomology complex for a pre-Lie algebra $(\g,\ast_\g)$ with a representation $(V;\rho,\mu)$ is given as follows (\cite{cohomology of pre-Lie}).
The set of $n$-cochains is given by
$C^n(\g,V):=\Hom(\wedge^{n-1}\g\otimes \g,V),\
n\geq 1.$  For all $\phi\in C^{n}(\g,V)$, the coboundary operator $\dM:C^{n}(\g,V)\longrightarrow C^{n+1}(\g,V)$ is given by
 \begin{eqnarray}\label{eq:pre-Lie cohomology}
 \dM\phi(x_1, \cdots,x_{n+1})
 \nonumber&=&\sum_{i=1}^{n}(-1)^{i+1}\rho(x_i)\phi(x_1, \cdots,\hat{x_i},\cdots,x_{n+1})\\
\nonumber &&+\sum_{i=1}^{n}(-1)^{i+1}\mu(x_{n+1})\phi(x_1, \cdots,\hat{x_i},\cdots,x_n,x_i)\\
 \nonumber&&-\sum_{i=1}^{n}(-1)^{i+1}\phi(x_1, \cdots,\hat{x_i},\cdots,x_n,x_i\ast_\g x_{n+1})\\
\label{eq:cobold} &&+\sum_{1\leq i<j\leq n}(-1)^{i+j}\phi([x_i,x_j]_\g,x_1,\cdots,\hat{x_i},\cdots,\hat{x_j},\cdots,x_{n+1}),
\end{eqnarray}
for all $x_i\in \g,~i=1,\cdots,n+1$. In particular, we use the
symbol $\dM^{\reg}$ to refer the coboundary operator  associated to the regular representation. We denote the
$n$-th cohomology group for the coboundary operator $\dM^{\reg}$  by
$H_{\rm reg}^n(\g,\g)$ and $H_{\rm reg}(\g,\g)=\oplus_{n\geq1}H_{\rm reg}^n(\g,\g)$.

\section{$F$-manifold algebras and pre-Lie deformations of commutative associative algebras}\label{sec:F-algebras and deformations}

\subsection{$F$-manifold algebras and $F$-manifold-admissible algebras}

In this subsection, we introduce the notion of $F$-manifold-admissible algebras, which give rise to $F$-manifold algebras. Various examples are given.
\begin{defi}{\rm (\cite{Dot})}
An {\bf $F$-manifold algebra} is a triple $(A,\cdot_A,[-,-]_A)$, where $(A,\cdot_A)$ is a commutative associative algebra and $(A,[-,-]_A)$ is a Lie algebra, such that for all $x,y,z,w\in A$, the Hertling-Manin relation holds:
\begin{equation}\label{eq:HM relation}
P_{x\cdot_A y}(z,w)=x\cdot_A P_{y}(z,w)+y\cdot_A P_{x}(z,w),
\end{equation}
where $P_{x}(y,z)$ is defined by
\begin{equation}
 P_{x}(y,z)=[x,y\cdot_A z]_A-[x,y]_A\cdot_A z-y\cdot_A [x,z]_A.
\end{equation}
\end{defi}

\begin{rmk}
 For any elements $z,w\in A$, the Hertling-Manin relation means that the linear map $P_{-}(z,w):A\longrightarrow A$ defined by $P_{-}(z,w)(x):=P_x(z,w)$ is a Hochschild $2$-cocycle of the commutative associative algebra $(A,\cdot_A)$.
\end{rmk}

\begin{rmk}
  An $F$-manifold is a pair $(M,\cdot)$, where $M$ is a manifold, $\cdot$ is a $\CWM$-bilinear, commutative,
associative multiplication on the tangent bundle $TM$, such that the Hertling-Manin relation holds, where the Lie bracket in the Hertling-Manin relation is the usual Lie bracket of vector fields. Thus the underlying algebraic structure of an $F$-manifold (ignore the geometry of the arguments) is an $F$-manifold algebra.
\end{rmk}

\begin{defi}
Let $(A,\cdot_A,[-,-]_A)$ and $(B,\cdot_B,[-,-]_B)$ be two $F$-manifold algebras. A {\bf homomorphism} between $A$ and $B$ is a linear map $\varphi:A\rightarrow B$ such that
\begin{eqnarray}
 \varphi(x\cdot_A y)&=&\varphi(x)\cdot_B\varphi(y),\\
  \varphi[x,y]_A&=&[\varphi(x),\varphi(y)]_B,\quad \forall~x,y\in A.
\end{eqnarray}
\end{defi}

\begin{ex}\label{ex:direct sum of HMA}{\rm
  Let $(A,\cdot_A,[-,-]_A)$ and $(B,\cdot_B,[-,-]_B)$ be two $F$-manifold algebras. Then $(A\oplus B,\cdot_{A\oplus B},[-,-]_{A\oplus B})$ is an $F$-manifold algebra, where the product $\cdot_{A\oplus B}$ and bracket $[-,-]_{A\oplus B}$ are given by
  \begin{eqnarray*}
   ( x_1+ x_2) \cdot_{A\oplus B} (y_1+ y_2)&=& x_1\cdot_A y_1+x_2\cdot_B y_2,\\
   {[  ( x_1+ x_2), (y_1+ y_2)]_{A\oplus B}}&=&[x_1,y_1]_A+ [x_2, y_2]_B
  \end{eqnarray*}
  for all $x_1,y_1\in A,x_2,y_2\in B.$}
\end{ex}

\begin{ex}\label{ex:tensor of HMA}{\rm
  Let $(A,\cdot_A,[-,-]_A)$ and $(B,\cdot_B,[-,-]_B)$ be two $F$-manifold algebras. Then $(A\otimes B,\cdot_{A\otimes B},[-,-]_{A\otimes B})$ is an $F$-manifold algebra, where the product $\cdot_{A\otimes B}$ and bracket $[-,-]_{A\otimes B}$ are given by
  \begin{eqnarray*}
   ( x_1\otimes x_2) \cdot_{A\otimes B} (y_1\otimes y_2)&=& (x_1\cdot_A y_1) \otimes (x_2\cdot_B y_2),\\
   {[ x_1\otimes x_2, y_1\otimes y_2]_{A\otimes B}}&=&[x_1,y_1]_A \otimes (x_2\cdot_B y_2)+(x_1\cdot_A y_1) \otimes [x_2,y_2]_B
  \end{eqnarray*}
  for all $x_1,y_1\in A,x_2,y_2\in B.$}
\end{ex}

Recall that a {\bf Poisson algebra} is a triple $(P,\cdot_P,\{-,-\}_P)$, where $(P,\cdot_P)$ is a commutative associative algebra and $(P,\{-,-\}_P)$ is a Lie algebra, such that the Leibniz rule holds:
$$\{x,y\cdot_P z\}_P=\{x,y\}_P\cdot_P z+y\cdot_P\{x,z\}_P,\quad \forall~ x,y,z\in P.$$

\begin{ex}{\rm
Any Poisson algebra is an $F$-manifold algebra.}
\end{ex}

The notion of PreLie-Com algebras was given in \cite{Foissy,Mansuy}.
\begin{defi}
A {\bf pre-Lie commutative algebra (or PreLie-Com algebra)} is a triple $(A,\cdot_A,\ast_A)$, where $(A,\cdot_A)$ is a commutative associative algebra and $(A,\ast_A)$ is a pre-Lie algebra satisfying
\begin{equation}
  x\ast_A(y\cdot_A z)-(x\ast_A y)\cdot_A z-y\cdot_A(x\ast_A z)=0,\quad\forall~x,y,z\in A.
\end{equation}
\end{defi}

\begin{cor}{\rm(\cite{Dot})}
  Let $(A,\cdot_A,\ast_A)$ be a PreLie-Com algebra. Then $(A,\cdot_A,[-,-]_A)$ is an $F$-manifold algebra, where the bracket $[-,-]_A$ is given by
\begin{equation}\label{eq:pseudo-bracket}
[x,y]_A=x\ast_A y-y\ast_A x,\quad\forall~x,y\in A.
\end{equation}
\end{cor}

Let $A=\K[x_1,x_2,\cdots,x_n]$ be the algebra of polynomials in $n$ variables. Let $$\huaD_n=\{\partial_{x_1},\partial_{x_2},\ldots,\partial_{x_n}\}$$ be the system of derivations over $A$. For any  polynomial $f\in A$, the endomorphisms
  $$f\partial_{x_i}:A\longrightarrow A,\quad (f\partial_{x_i})(g)=f\partial_{x_i}(g),\quad \forall~g\in A$$
  are derivations of $A$. Denote by $A\huaD_n=\{\sum_{i=1}^nf_i\partial_{x_i}\mid f_i\in A,\partial_{x_i}\in\huaD_n\}$ the space of derivations.

\begin{ex}\label{ex:F-algebra-poly}\label{ex:poly}
{\rm
  Let $A$ be the algebra of polynomials in $n$ variables. Define $\cdot:A\huaD_n\times A\huaD_n\longrightarrow A\huaD_n $ and $\ast:A\huaD_n\times A\huaD_n\longrightarrow A\huaD_n$ by
  \begin{eqnarray*}
  (f\partial_{x_i})\cdot (g\partial_{x_j})&=&(fg)\delta_{ij}\partial_{x_i},\\
  (f\partial_{x_i})\ast (g\partial_{x_j})&=&f\partial_{x_i}(g)\partial_{x_j},\quad\forall~f,g\in A.
  \end{eqnarray*}
  Then $(A\huaD_n,\cdot,\ast)$ is a PreLie-Com algebra. Furthermore, $(A\huaD_n,\cdot,[-,-])$ is an $F$-manifold algebra, where the bracket is given by
  $$ [f\partial_{x_i},g\partial_{x_j}]=f\partial_{x_i}(g)\partial_{x_j}-g\partial_{x_j}(f)\partial_{x_i},\quad\forall~f,g\in A.$$}
\end{ex}

\begin{ex}\rm{
 Let $A=\K[x_1,x_2]$ be the algebra of polynomials in two variables. Besides the PreLie-Com algebra structure given in Example~\ref{ex:poly} on
 $A\huaD_2$, there is another PreLie-Com algebra  $(A\huaD_2,\cdot,\ast)$, where the operations $\cdot$ and  $\ast$ are determined by
\begin{eqnarray*}
 \partial_{x_1}\cdot \partial_{x_1}=\partial_{x_1},~\quad
 \partial_{x_1}\cdot \partial_{x_2}&=& \partial_{x_2}\cdot \partial_{x_1}=\partial_{x_2},~
  \quad \partial_{x_2}\cdot \partial_{x_2}=0,\\
  (f\partial_{x_i})\ast (g\partial_{x_j})&=&f\partial_{x_i}(g)\partial_{x_j},\quad\forall~f,g\in A~(i,j=1,2).
  \end{eqnarray*}
Furthermore, $(A\huaD_2,\cdot,[-,-])$ is an $F$-manifold algebra, where the bracket is given by
  $$ [f\partial_{x_i},g\partial_{x_j}]=f\partial_{x_i}(g)\partial_{x_j}-g\partial_{x_j}(f)\partial_{x_i},\quad\forall~f,g\in A.$$}
\end{ex}

In the sequel, we introduce the notion of   $F$-manifold-admissible algebras, which include PreLie-Com algebras as special cases, and can be used to construct $F$-manifold algebras.
\begin{defi}\label{defi:pseudo-pre-HMA}
An {\bf $F$-manifold-admissible algebra} is a vector space $A$ equipped with two bilinear maps $\cdot_A:A\times A\rightarrow A$ and $\ast_A:A\times A\rightarrow A$ such that $(A,\cdot_A)$ is a commutative associative algebra and $(A,\ast_A)$ is a Lie-admissible algebra satisfying for all $x,y,z\in A$,
\begin{equation}\label{eq:pseudo-pre-HM1}
x\ast_A(y\cdot_A z)-(x\ast_A y)\cdot_A z-y\cdot_A(x\ast_A z)=y\ast_A(x\cdot_A z)-(y\ast_A x)\cdot_A z-x\cdot_A(y\ast_A z).
\end{equation}
\end{defi}

\begin{rmk}
Viewing the associative algebra $A$ as a pre-Lie algebra, the condition \eqref{eq:pseudo-pre-HM1} means that the $F$-manifold admissible operation $\ast_A:A\times A\rightarrow A$ is a $2$-cocycle of the pre-Lie algebra $(A,\cdot_A) $ with the coefficients in the regular representation, i.e. $\dM^{\reg}(\ast_A)=0$.
\end{rmk}

\begin{thm}\label{pro:pseudo-pre-HMA-HMA}
  Let $(A,\cdot_A,\ast_A)$ be an $F$-manifold-admissible algebra. Then $(A,\cdot_A,[-,-]_A)$ is an $F$-manifold algebra, where $[-,-]_A$ is given by \eqref{eq:pseudo-bracket}.
\end{thm}
\begin{proof}
  We only need to verify the Hertling-Manin relation. By \eqref{eq:pseudo-pre-HM1}, we have
  \begin{eqnarray*}
    P_{x}(y,z)&=&[x,y\cdot_A z]_A-[x,y]_A\cdot_A z-y\cdot_A [x,z]_A\\
    &=&x\ast_A (y\cdot_A z)-(y\cdot_A z)\ast_A x-(x\ast_A y)\cdot_A z+(y\ast_A x)\cdot_A z\\
    &&-y\cdot_A(x\ast_A z)+y\cdot_A(z\ast_A x)\\
    &=&x\ast_A (y\cdot_A z)-(x\ast_A y)\cdot_A z-y\cdot_A(x\ast_A z)-(y\cdot_A z)\ast_A x\\
    &&+(y\ast_A x)\cdot_A z+y\cdot_A(z\ast_A x)\\
    &=&y\ast_A (x\cdot_A z)-(y\ast_A x)\cdot_A z-x\cdot_A(y\ast_A z)-(y\cdot_A z)\ast_A x\\
    &&+(y\ast_A x)\cdot_A z+y\cdot_A(z\ast_A x).
  \end{eqnarray*}
  By this formula and \eqref{eq:pseudo-pre-HM1}, we have
  \begin{eqnarray*}
    P_{x\cdot_A y}(z,w)&=&z\ast_A ((x\cdot_A y)\cdot_A w)-(z\ast_A(x\cdot_A y))\cdot_A w-(x\cdot_A y)\cdot_A(z\ast_A w)-(z\cdot_A w)\ast_A (x\cdot_A y)\\
    &&+(z\ast_A (x\cdot_A y))\cdot_A w+z\cdot_A(w\ast_A (x\cdot_A y))\\
    &=&z\ast_A (x\cdot_A (y\cdot_A w))-\big((z\ast_A x)\cdot_A y-x\cdot_A(z\ast_A y)+x\ast_A(z\cdot_A y)-(x\ast_A z)\cdot_A y\\
    &&-z\cdot_A(x\ast_A y)\big)\cdot_A w-(x\cdot_A y)\cdot_A(z\ast_A w)-\big(((z\cdot_A w)\ast_A x)\cdot_A y+x\cdot_A((z\cdot_A w)\ast_A y)\\
    &&+x\ast_A((z\cdot_A w)\cdot_A y)-(x\ast_A (z\cdot_A w)\cdot_A y)-(z\cdot_A w)\cdot_A (x\ast_A y)\big)+\big((z\ast_A x)\cdot_A y\\
    &&+x\cdot_A(z\ast_A y)+x\ast_A(z\cdot_A y)-(x\ast_A z)\cdot_A y-z\cdot_A(x\ast_A y))\big)\cdot_A w+\big((w\ast_A x)\cdot_A y\\
    &&+x\cdot_A (w\ast_A y)+x\ast_A (w\cdot_A y)-(x\ast_A w)\cdot_A y-w\cdot_A(x\ast_A y)\big)\cdot_A z\\
    &=&y\cdot_A P_x(z,w)+x\cdot_A \big(-w\cdot_A (z\ast_A y)-y\cdot_A(z\ast_A w)-(z\cdot_A w)\ast_A y+(z\ast y)\cdot_A w\\
    &&+(w\ast_A y)\cdot_A z\big)+\big(z\ast_A (x\cdot_A (y\cdot_A w))-(z\ast_A x)\cdot_A (y\cdot_A w)\\
    &&-x\ast_A (z\cdot_A (y\cdot_A w))+(x\ast_A z)\cdot_A (y\cdot_A w)+z\cdot_A(x\ast_A(y\cdot_A w))\big)\\
    &=&y\cdot_A P_x(z,w)+x\cdot_A \big(-w\cdot_A (z\ast_A y)-y\cdot_A(z\ast_A w)-(z\cdot_A w)\ast_A y+(z\ast y)\cdot_A w\\
    &&+(w\ast_A y)\cdot_A z+\big)+x\cdot_A(z\ast_A (y\cdot_A w))\\
    &=&y\cdot_A P_x(z,w)+x\cdot_A P_y(z,w).
  \end{eqnarray*}
  This shows that the Hertling-Manin relation holds.
\end{proof}

\begin{pro}\label{ex:derivation-HMA}
 Let $(A,\cdot)$ be a commutative associative algebra with a derivation $D$. Then the new product
  \begin{eqnarray*}
    x\ast_a y&=&x\cdot D (y)+a\cdot x\cdot y,\quad\forall~x,y\in A
  \end{eqnarray*}
  makes $(A,\cdot,\ast_a)$ being an $F$-manifold-admissible algebra for any fixed $a\in \K$ or $a\in A$. In particular, for $a=0$, $(A,\cdot,\ast_0)$ is a PreLie-Com algebra. Furthermore, $(A,\cdot,[-,-])$ is an $F$-manifold algebra, where the bracket is given by
  $$ [x,y]=x\ast_a y-y\ast_a x=x\cdot D( y)-y\cdot D(x),\quad\forall~x,y\in A.$$
\end{pro}
\begin{proof}
  It was shown that $(A,\ast_a)$ is a pre-Lie algebra for $a=0$ by S.I. Gel'fand \cite{GD}, for $a\in \K$ by Filippov \cite{Fil} and for $a\in A$ by Xu \cite{Xu}. Thus $(A,\ast_a)$ is a Lie-admissible algebra. Furthermore, by the fact that $(A,\cdot)$ is a commutative associative algebra and  $D$ is a derivation on it, we have
  \begin{eqnarray*}
   && x\ast_a(y\cdot z)-(x\ast_a y)\cdot z-y\cdot(x\ast_a z)\\
    &=&x\cdot D (y\cdot z)+x\cdot(a\cdot y\cdot z)-(x\cdot D (y)+a\cdot x\cdot y)\cdot z-y\cdot (x\cdot D (z)+a\cdot x\cdot z)\\
    &=&-a\cdot x\cdot y\cdot z.
  \end{eqnarray*}
  Similarly, we have
  $$y\ast_a(x\cdot z)-(y\ast_a x)\cdot z-x\cdot(y\ast_a z)=-a\cdot x\cdot y\cdot z.$$
  Thus $$x\ast_a(y\cdot z)-(x\ast_a y)\cdot z-y\cdot(x\ast_a z)=y\ast_a(x\cdot z)-(y\ast_a x)\cdot z-x\cdot(y\ast_a z),$$
  which implies that $(A,\cdot,\ast_a)$ is an $F$-manifold-admissible algebra for any fixed $a\in \K$ or $a\in A$. It is obvious that $(A,\cdot,\ast_0)$ is a PreLie-Com algebra. By Theorem \ref{pro:pseudo-pre-HMA-HMA}, $(A,\cdot,[-,-])$ is an $F$-manifold algebra.
\end{proof}

\begin{ex}
{\rm  Let $A$ be the algebra of polynomials in $n$ variables. Set $D=\sum_{i=1}^nf_i\partial_{x_i}\in A\huaD_n$ for $f_i\in A,\partial_{x_i}\in\huaD_n$.
 Fixed a polynomial $\kappa\in A$, define $\cdot:A\times A\longrightarrow A$ and  $\ast_\kappa:A\times A\longrightarrow A$ by
  \begin{eqnarray*}
  g\cdot h&=& gh,\\
g\ast_\kappa h&=&gD(h)+\kappa g h=\sum_{i=1}^n gf_i\partial_{x_i}(h)+\kappa g h,\quad\forall~g,h\in A.
  \end{eqnarray*}
  Then by Proposition \ref{ex:derivation-HMA}, $(A,\cdot,\ast_\kappa)$ is an $F$-manifold-admissible algebra. Furthermore, $(A,\cdot,[-,-])$ is an $F$-manifold algebra, where the bracket is given by
  $$ [g,h]=\sum_{i=1}^n f_i\big(g\partial_{x_i}(h)-h\partial_{x_i}(g)\big),\quad\forall~g,h\in A.$$}
\end{ex}

\begin{ex}{\rm
  Let $A$ be a $2$-dimensional vector space  with basis $\{e_1,e_2\}$. Then $A$ with the non-zero multiplication
  \begin{eqnarray*}
  e_1\cdot e_1&=&e_1,\quad e_1\cdot e_2=e_2\cdot e_1=e_2
  \end{eqnarray*}
  is a commutative associative algebra. It is straightforward to check that all derivations on $(A,\cdot)$ are determined by
  $$D(e_2)=a e_2,\quad\forall~a\in \K.$$
  Thus by Proposition \ref{ex:derivation-HMA}, $(A,\cdot,[-,-])$ with the bracket
  $$[e_1,e_2]=a e_2$$
  is an $F$-manifold algebra.

  }
\end{ex}

\begin{ex}\label{ex:3-dimensional F-algebra}{\rm
  Let $A$ be a $3$-dimensional vector space  with basis $\{e_1,e_2,e_3\}$. Then $A$ with the non-zero multiplication
  \begin{eqnarray*}
 e_2\cdot e_3=e_3\cdot e_2=e_1,\quad  e_3\cdot e_3=e_2
  \end{eqnarray*}
  is a commutative associative algebra. It is straightforward to check that all derivations on $(A,\cdot)$ are determined by
 \begin{eqnarray*}
  D(e_1)=3a e_1,\quad D(e_2)=2b e_1+2a e_2,\quad
  D(e_3)=ce_1+be_2+ae_3,\quad \forall~a,b,c\in \K.
  \end{eqnarray*}
  Thus by Proposition \ref{ex:derivation-HMA}, $(A,\cdot,[-,-])$ with the bracket
  $$[e_2,e_3]=-a e_1$$
  is an $F$-manifold algebra.

  }
\end{ex}

\subsection{Pre-Lie deformation quantization of commutative pre-Lie algebras} In this subsection, we introduce the notion of   pre-Lie formal deformations of commutative associative algebras (that is,  commutative pre-Lie algebras) and show that $F$-manifold algebras are the corresponding semi-classical limits. This illustrates that $F$-manifold algebras are the same to pre-Lie algebras as Poisson algebras to associative algebras. Furthermore, we show that pre-Lie infinitesimal deformations and extensions of pre-Lie $n$-deformations to pre-Lie $(n+1)$-deformations of a commutative associative algebra $A$ are classified by the second and the third cohomology group of the pre-Lie algebra $A$ (view the commutative associative algebra $A$ as a pre-Lie algebra).

Let $(A,\cdot_A)$ be a commutative associative algebra. Recall that an {\bf associative formal deformation}  of $A$ is a sequence of bilinear maps $\mu_n:A\times A\rightarrow A$ for $n\geqslant 0$ with $\mu_0$ being the commutative associative algebra product $\cdot_A$ on $A$, such that the $\K[[h]]$-bilinear product $\cdot_\hbar$ on $A[[\hbar]]$ determined by
$$x\cdot_\hbar y=\sum_{n=0}^\infty\hbar^n\mu_n(x,y),\quad\forall~x,y\in A$$
is associative, where $A[[\hbar]]$ is the set of formal power series of $\hbar$ with coefficients in $A$. Define
$$\{x,y\}=\mu_1(x,y)-\mu_1(y,x),\quad\forall~x,y\in A.$$
It is well known that $(A,\cdot_A,\{-,-\})$ is a Poisson algebra, called the {\bf semi-classical limit} of $(A[[\hbar]],\cdot_\hbar)$.

Since an associative algebra can be regarded as a pre-Lie algebra,
one may look for formal deformations of a commutative associative
algebra into pre-Lie algebras, that is, in the aforementioned
associative formal deformation, replace the associative product
$\cdot_{\bar h}$ by the pre-Lie product, and wonder what
additional structure will appear on $A$. On the other hand, such
an approach can be also seen as  formal deformations of a
commutative pre-Lie algebra into (non-commutative) pre-Lie
algebras, which is completely parallel to the associative formal
deformations of a commutative associative algebra  into
(non-commutative) associative algebras. Surprisingly, we find
that this is the structure of the $F$-manifold algebra. Now we
give the definition of a pre-Lie formal deformation of a
commutative associative algebra.
\begin{defi}
  Let $(A,\cdot_A)$ be a commutative associative algebra. A {\bf pre-Lie formal deformation} of $A$ is a sequence of bilinear maps $\mu_k:A\times A\rightarrow A$ for $k\geqslant 0$ with $\mu_0$ being the commutative associative algebra product $\cdot_A$ on $A$, such that the $\K[[h]]$-bilinear product $\cdot_\hbar$ on $A[[\hbar]]$ determined by
$$x\cdot_\hbar y=\sum_{n=0}^\infty\hbar^n\mu_n(x,y),\quad\forall~x,y\in A$$
is a pre-Lie algebra product.
\end{defi}
Note that the rule of pre-Lie algebra product $\cdot_\hbar$ on $A[[\hbar]]$ is equivalent to
  \begin{equation}\label{eq:pre-Lie rule}
  \sum_{i+j=k}\big(\mu_i(\mu_j(x,y),z)-\mu_i(x,\mu_j(y,z))\big)=\sum_{i+j=k}\big(\mu_i(\mu_j(y,x),z)-\mu_i(y,\mu_j(x,z))\big),\quad\forall~k\geqslant 0.
  \end{equation}

\begin{thm}\label{thm:deformation quantization}
Let $(A,\cdot_A)$ be a commutative associative algebra and $(A[[\hbar]],\cdot_\hbar)$ a pre-Lie formal deformation of $A$. Define
$$[x,y]_A=\mu_1(x,y)-\mu_1(y,x),\quad\forall~x,y\in A.$$
Then $(A,\cdot_A,[-,-]_A)$ is an $F$-manifold algebra. The $F$-manifold algebra $(A,\cdot_A,[-,-]_A)$ is called the {\bf semi-classical limit} of $(A[[\hbar]],\cdot_\hbar)$. The pre-Lie algebra $(A[[\hbar]],\cdot_\hbar)$ is called a {\bf pre-Lie deformation quantization} of $(A,\cdot_A)$.
\end{thm}

\begin{proof}
  Define the bracket $[-,-]_\hbar$ on $A[[\hbar]]$ by
$$[x,y]_\hbar=x\cdot_\hbar y-y\cdot_\hbar x=\hbar[x,y]_A+\hbar^2(\mu_2(x,y)-\mu_2(y,x))+\cdots, \quad \forall~x,y\in A.$$
By the fact that $(A[[\hbar]],\cdot_\hbar)$ is a pre-Lie algebra, $(A[[\hbar]],[-,-]_\hbar)$ is a Lie algebra. The $\hbar^2$-terms of the Jacobi identity for $[-,-]_\hbar$ gives the Jacobi identity for $[-,-]_A$. Thus $(A,[-,-]_A)$ is a Lie algebra.

For $k=1$ in \eqref{eq:pre-Lie rule}, by the commutativity of $\mu_0$, we have
\begin{eqnarray*}
  \mu_0(\mu_1(x,y),z)-\mu_0(x,\mu_1(y,z))-\mu_1(x,\mu_0(y,z))=\mu_0(\mu_1(y,x),z)-\mu_0(y,\mu_1(x,z))-\mu_1(y,\mu_0(x,z)).
\end{eqnarray*}
This is just the equality \eqref{eq:pseudo-pre-HM1} with $x\cdot_A y=\mu_0(x,y)$ and $x\ast_A y=\mu_1(x,y)$ for $x,y\in A$.
Thus $(A,\cdot_A,\ast_A)$ is an $F$-manifold-admissible algebra. By Theorem \ref{pro:pseudo-pre-HMA-HMA}, $(A,\cdot_A,[-,-]_A)$ is an $F$-manifold algebra.
\end{proof}

In the sequel, we study pre-Lie $n$-deformations and pre-Lie infinitesimal deformations of   commutative associative algebras.
\begin{defi}
  Let $(A,\cdot_A)$ be a commutative associative algebra. A {\bf pre-Lie  $n$-deformation} of $A$ is a sequence of bilinear maps $\mu_i:A\times A\rightarrow A$ for $0\leq i\leq n$ with $\mu_0$ being the commutative associative algebra product $\cdot_A$ on $A$, such that the $\K[[h]]/(\hbar^{n+1})$-bilinear product $\cdot_\hbar$ on $A[[\hbar]]/(\hbar^{n+1})$ determined by
$$x\cdot_\hbar y=\sum_{k=0}^n\hbar^n\mu_k(x,y),\quad\forall~x,y\in A$$
is a pre-Lie algebra product.
\end{defi}

We call a pre-Lie $1$-deformation of a commutative associative algebra $(A,\cdot_A)$ a {\bf pre-Lie infinitesimal deformation} and denote it by $(A,\mu_1)$.

By direct calculations, $(A,\mu_1)$ is a pre-Lie infinitesimal deformation of a commutative associative algebra $(A,\cdot_A)$  if and only if for all $x,y,z\in A$
 \begin{eqnarray}
 \label{2-closed} \mu_1(x,y)\cdot_A z-x\cdot_A \mu_1(y,z)-\mu_1(x,y\cdot_A z)&=&\mu_1(y,x)\cdot_A z-y\cdot_A \mu_1(x,z)-\mu_1(y,x\cdot_A z).
\end{eqnarray}
Equation $(\ref{2-closed})$ means that $\mu_1$ is a $2$-cocycle for the pre-Lie algebra $(A,\cdot_A)$, i.e. $\dM^{\reg}\mu_1=0$.

Two pre-Lie infinitesimal deformations  $A_\hbar=(A,\mu_1)$ and $A'_\hbar=(A,\mu'_1)$  of a
commutative associative algebra $(A,\cdot_A)$ are said to be {\bf
equivalent} if there exists a family of pre-Lie algebra
homomorphisms ${\Id}+\hbar\varphi:A_\hbar\longrightarrow A'_\hbar$ modulo $\hbar^2$. A pre-Lie infinitesimal deformation
is said to be {\bf trivial} if there exists a family of pre-Lie
algebra homomorphisms ${\Id}+\hbar\varphi:A_\hbar\longrightarrow (A,\cdot_A)$ modulo $\hbar^2$.

By direct calculations, $A_\hbar$ and $A'_\hbar$  are
equivalent pre-Lie infinitesimal deformations if and only if
\begin{eqnarray}
\mu_1(x,y)-\mu_1'(x,y)&=&x\cdot_A \varphi(y)+\varphi(x)\cdot_A y-\varphi(x\cdot_A y).\label{2-exact}
\end{eqnarray}
Equation $(\ref{2-exact})$ means that $\mu_1-\mu_1'=\dM^{\reg}\varphi$. Thus we have

\begin{thm}
 There is a one-to-one correspondence between the space of equivalence classes of pre-Lie infinitesimal deformations of $A$ and the second cohomology group $H_{\rm reg}^2(A,A)$.
\end{thm}

It is routine to check that

\begin{pro}
  Let $(A,\cdot_A)$ be a commutative associative algebra such that $H_{\rm reg}^2(A,A)=0$. Then all pre-Lie infinitesimal deformations of $A$ are trivial.
\end{pro}

\begin{defi}
Let $\{\mu_1, \cdots,\mu_{n}\}$ be a pre-Lie $n$-deformation of a commutative associative algebra $(A,\cdot_A)$. A pre-Lie $(n+1)$-deformation of a commutative associative algebra $(A,\cdot_A)$ given by $\{\mu_1,\cdots,\mu_n,\mu_{n+1}\}$ is called an {\bf extension} of the pre-Lie $n$-deformation given by $\{\mu_1, \cdots,\mu_{n}\}$.
\end{defi}

\begin{thm}
  For any pre-Lie $n$-deformation of a commutative associative algebra $(A,\cdot_A)$, the $\Theta_n\in\Hom(\otimes^3\g,\g)$ defined by
\begin{equation}\label{eq:3-cocycle}
\Theta_n(x,y,z)=\sum_{i+j=n+1,i,j\geq 1}\big(\mu_i(\mu_j(x,y),z)-\mu_i(x,\mu_j(y,z))-\mu_i(\mu_j(y,x),z)+\mu_i(y,\mu_j(x,z))\big).
\end{equation}
is a cocycle, i.e. $\dM^{\reg}\Theta_n=0$.

 Moreover, the pre-Lie $n$-deformation $\{\mu_1, \cdots,\mu_{n}\}$ extends into some pre-Lie $(n + 1)$-deformation if and only if $[\Theta_n]=0$ in $H_{\rm reg}^3(A,A)$.
\end{thm}
\begin{proof}
It is obvious that
\begin{equation*}
\Theta_n(x,y,z)=-\Theta_n(y,x,z),\quad \forall~x,y,z\in A.
\end{equation*}
Thus $\Theta_n$ is an element of $C^3(A,A)$. It is straightforward to check that the cochain $\Theta_n\in C^3(A,A)$ is closed.

Assume that the pre-Lie $(n+1)$-deformation of a commutative associative algebra $(A,\cdot_A)$   given by $\{\mu_1,\cdots,\mu_n,\mu_{n+1}\}$ is an extension of the pre-Lie $n$-deformation   given by $\{\mu_1,\cdots,\mu_n\}$, then we have
\begin{eqnarray*}
  \label{eq:n+1 term}&&x\cdot_A \mu_{n+1}(y,z)-y\cdot_A \mu_{n+1}(x,z)+ \mu_{n+1}(y,x)\cdot_A z-\mu_{n+1}(x,y)\cdot_A z+\mu_{n+1}(y,x)\cdot_A z \\
   \nonumber&&\quad-\mu_{n+1}(x,y)\cdot_A z=\sum_{i+j=n+1,i,j\geq 1}\big(\mu_i(\mu_j(x,y),z)-\mu_i(x,\mu_j(y,z))-\mu_i(\mu_j(y,x),z)+\mu_i(y,\mu_j(x,z))\big).
\end{eqnarray*}
It is obvious that the right-hand side of the above equality is just $\Theta_n(x,y,z)$. We can rewrite the above equality as
$$\dM^{\reg}\mu_{n+1}(x,y,z)=\Theta_n(x,y,z).$$
We conclude that, if a pre-Lie $n$-deformation of a commutative associative algebra $(A,\cdot_A)$ extends to a pre-Lie $(n + 1)$-deformation, then $\Theta_n$ is coboundary.

Conversely, if $\Theta_n$ is coboundary, then there exists an element $\psi\in C^2(A,A)$ such that
$$\dM^{\reg}\psi(x,y,z)=\Theta_n(x,y,z).$$
It is not hard to check that $\{\mu_1,\cdots,\mu_n,\mu_{n+1}\}$ with $\mu_{n+1}=\psi$ generates a  pre-Lie $(n+1)$-deformation of $(A,\cdot_A)$ and thus this pre-Lie $(n+1)$-deformation is an extension of the pre-Lie $n$-deformation given by $\{\mu_1, \cdots,\mu_{n}\}$.
\end{proof}

\section{Pre-$F$-manifold algebras, Rota-Baxter operators and  $\huaO$-operators on $F$-manifold algebras}\label{sec:pre-F-algebras}
In this section, first we study representations of an $F$-manifold algebra. Then we introduce the notions of pre-$F$-manifold algebras and Rota-Baxter operators (more generally $\huaO$-operators) on an $F$-manifold algebra. We show that on one hand, an $\huaO$-operator on an $F$-manifold algebra gives a pre-$F$-manifold algebra, and on the other hand, a pre-$F$-manifold algebra naturally gives an $\huaO$-operator on the sub-adjacent $F$-manifold algebra. More examples on pre-$F$-manifold algebras and $F$-manifold algebras are given.

\subsection{Representations of $F$-manifold algebras}
In this subsection, we introduce the notion of representations of $F$-manifold algebras.

Let $(A,\cdot_A)$ be a commutative associative algebra. Recall that a {\bf representation} of $A$ on a vector space $V$ is a linear map $\mu:A\longrightarrow\gl(V)$ such that $\mu(x\cdot_A y)=\mu(x)\circ\mu(y)$ for any $x,y\in A$. We will denote a representation of $A$ by $(V;\mu)$.
 Let $(V;\mu)$ be a representation of a commutative associative algebra $(A,\cdot_A)$.
   Define $\mu^*:A\longrightarrow \gl(V^*)$ by
$$
 \langle \mu^*(x)\alpha,v\rangle=-\langle \alpha,\mu(x)v\rangle,\quad \forall ~ x\in A,\alpha\in V^*,v\in V.
$$
  Then $(V^*;-\mu^*)$ is a representation of $(A,\cdot_A)$.

\begin{ex}{\rm
  Let $(A,\cdot_A)$ be a commutative associative algebra. Let $\huaL_x$ denote the multiplication operator, that is,
 $
  \huaL_xy=x\cdot_A y,$ for all $~x,y\in A.
 $
   Then $(A;\huaL)$ is a representation of $(A,\cdot_A)$, called the {\bf regular representation}. Furthermore,  $(A^*;-\huaL^*)$ is also a representation of $(A,\cdot_A)$.
  }
\end{ex}

 Similarly, let $(\frak g,[-,-]_{\frak g})$ be a Lie algebra  and $V$ a vector space. Let $\rho:\frak g\rightarrow
\frak g\frak l(V)$ be a linear map. The pair $(V;\rho)$ is called a {\bf representation} of $\frak g$ if for all
$x,y\in \frak g$, we have
$\rho([x,y]_{\frak g})=[\rho(x),\rho(y)].$
Let $(V;\rho)$ be a representation of a Lie algebra $(\frak g,[-,-]_{\frak g})$.
   Define $\rho^*:\frak g\longrightarrow \gl(V^*)$  by
$$
 \langle \rho^*(x)(\alpha),v\rangle=-\langle \alpha,\rho(x) (v)\rangle,\quad \forall~x\in \frak g,\alpha\in V^*,v\in V.
$$
  Then $(V^*;\rho^*)$ is a representation of $(\g, [-,-]_\g)$.
  \begin{ex}
 Define $\ad:\g\lon\gl(\g)$ by $\ad_xy=[x,y]_\g$ for all $x,y\in \g$. Then   $(\frak g;\ad)$ is a representation of $(\g, [-,-]_\g)$, called the {\bf adjoint representation}. Furthermore,  $(\g^*;\ad^*)$ is also a representation of $(\g, [-,-]_\g)$.
  \end{ex}

\begin{defi}Let $(A,\cdot_A,[-,-]_A)$ be an $F$-manifold algebra. A {\bf representation} of $A$ is a triple $(V;\rho,\mu)$ such that $(V;\rho)$ is a  representation of the Lie algebra $(A,[-,-]_A)$ and $(V;\mu)$ is a representation of the commutative associative algebra $(A,\cdot_A)$ satisfying
   \begin{eqnarray}
     \label{eq:rep 1}R_{\rho,\mu}(x\cdot_A y,z)=\mu(x)\circ R_{\rho,\mu}(y,z)+\mu(y)\circ R_{\rho,\mu}(x,z),\\
     \label{eq:rep 2} \mu(P_x(y,z))=S_{\rho,\mu}(y,z)\circ\mu(x)-\mu(x)\circ S_{\rho,\mu}(y,z),
   \end{eqnarray}
   where $R_{\rho,\mu},S_{\rho,\mu}: A\otimes A\rightarrow \gl(V)$ are defined by
   \begin{eqnarray}
  \label{eq:repH 1}R_{\rho,\mu}(x,y)&=&\rho(x)\circ\mu(y)-\mu(y)\circ \rho(x)-\mu([x,y]_A),\\
  \label{eq:repH 2}S_{\rho,\mu}(x,y)&=&\mu(x)\circ\rho(y)+\mu(y)\circ \rho(x)-\rho(x\cdot_A y)
   \end{eqnarray}
  for all $x,y,z\in A$.
  \end{defi}

It is straightforward to obtain the following conclusion.
\begin{pro}\label{pro:semi-direct}
 Let $(A,\cdot_A,[-,-]_A)$ be an $F$-manifold algebra and $(V;\rho,\mu)$ a representation. Then $(A\oplus V,\cdot_{\mu},[-,-]_\rho)$ is an $F$-manifold algebra, where $(A\oplus V,\cdot_{\mu})$ is the semi-direct product commutative associative algebra $A\ltimes_{\mu} V$, i.e.
 $$
  (x_1+v_1)\cdot_{\mu}(x_2+v_2)=x_1\cdot_A x_2+\mu(x_1)v_2+\mu(x_2)v_1,\quad \forall~ x_1,x_2\in A,~v_1,v_2\in V
$$ and $(A\oplus V,[-,-]_\rho)$ is the semi-direct product Lie algebra $A\ltimes_{\rho} V$, i.e. $$
  [x_1+v_1,x_2+v_2]_\rho=[x_1,x_2]_{\frak g}+\rho(x_1)(v_2)-\rho(x_2)(v_1),\quad \forall~x_1,x_2\in A,~v_1,v_2\in V.
$$
\end{pro}

\begin{ex}{\rm
 Let $(V;\rho,\mu)$ be a representation of a Poisson algebra $(P,\cdot_P,\{-,-\}_P)$, i.e. $(V;\rho)$ is a  representation of the Lie algebra $(P,\{-,-\}_P)$ and $(V;\mu)$ is a representation of the commutative associative algebra $(P,\cdot_P)$ satisfying
   \begin{eqnarray}
     \label{eq:Poisson rep 1}\rho(x)\circ\mu(y)-\mu(y)\circ \rho(x)-\mu([x,y]_A)&=&0,\\
     \label{eq:Poisson rep 2} \mu(x)\circ\rho(y)+\mu(y)\circ \rho(x)-\rho(x\cdot_A y)&=&0,\quad \forall~x,y,z\in P.
   \end{eqnarray}
  Then $(V;\rho,\mu)$ is also a representation of the $F$-manifold algebra given by this Poisson algebra $P$.}
\end{ex}

Let $(V;\rho,\mu)$ be a representation of a Poisson algebra $(P,\cdot_P,\{-,-\}_P)$. Then the triple $(V^*;\rho^*,-\mu^*)$ is also a representation of $P$. But  $F$-manifold algebras do not have this property. In fact, we have

\begin{pro}
  Let $(A,\cdot_A,[-,-]_A)$ be an $F$-manifold algebra. If the triple $(V;\rho,\mu)$, where  $(V;\rho)$ is a  representation of the Lie algebra $(A,[-,-]_A)$ and $(V;\mu)$ is a representation of the commutative associative algebra $(A,\cdot_A)$, satisfies
  \begin{eqnarray}
     \label{eq:corep 1}R_{\rho,\mu}(x\cdot_A y,z)= R_{\rho,\mu}(y,z)\circ\mu(x)+ R_{\rho,\mu}(x,z)\circ \mu(y),\\
     \label{eq:corep 2} \mu(P_x(y,z))=T_{\rho,\mu}(y,z)\circ\mu(x)-\mu(x)\circ T_{\rho,\mu}(y,z),
   \end{eqnarray}
   where $R_{\rho,\mu}$ is given by \eqref{eq:repH 1} and $T_{\rho,\mu}: A\otimes A\rightarrow \gl(V)$ is defined by
   \begin{eqnarray}
    \label{eq:repH 3} T_{\rho,\mu}(x,y)&=&\rho(y)\circ\mu(x)+\rho(x)\circ\mu(y)-\rho(x\cdot_A y),\quad\forall~x,y\in A,
   \end{eqnarray}
   then $(V^*;\rho^*,-\mu^*)$ is  a representation of $A$.
\end{pro}
\begin{proof}
By direct calculations, for all $x,y\in A,v\in V,\alpha\in V^*$, we have
\begin{eqnarray*}
  \langle R_{\rho^*,-\mu^*}(x,y)(\alpha),v\rangle=\langle\alpha, R_{\rho,\mu}(v)\rangle,\quad \langle S_{\rho^*,-\mu^*}(x,y)(\alpha),v\rangle=\langle\alpha, T_{\rho,\mu}(v)\rangle.
\end{eqnarray*}
Furthermore, we have
\begin{eqnarray*}
 &&\langle R_{\rho^*,-\mu^*}(x\cdot_A y,z)(\alpha)+\mu^*(x) R_{\rho^*,-\mu^*}(y,z)(\alpha)+\mu^*(y) R_{\rho^*,-\mu^*}(x,z)(\alpha),v\rangle\\
 &=&\langle \alpha,R_{\rho,\mu}(x\cdot_A y,z)(v)- R_{\rho,\mu}(y,z)\mu(x)(v)- R_{\rho,\mu}(x,z) \mu(y)(v)\rangle
\end{eqnarray*}
and
\begin{eqnarray*}
 &&\langle -\mu^*(P_x(y,z))(\alpha)+S_{\rho^*,-\mu^*}(y,z)\mu^*(x)(\alpha)+\mu^*(x) S_{\rho^*,-\mu^*}(y,z)(\alpha),v\rangle\\
 &=&\langle \alpha,\mu(P_x(y,z))(v)-T_{\rho,\mu}(y,z)\mu(x)(v)+\mu(x) T_{\rho,\mu}(y,z)(v)\rangle.
\end{eqnarray*}
By the hypothesis and the definition of representation, the conclusion follows immediately.
\end{proof}

\begin{ex}\label{ex:dual representation}{\rm
  Let $(A,\cdot_A,[-,-]_A)$ be an $F$-manifold algebra. Then $(A;\ad,\huaL)$ is a representation of $A$, which is also called the {\bf regular representation}. Furthermore, if the $F$-manifold algebra also satisfies the following relations:
  \begin{eqnarray}
   \label{eq:coh1} P_{x\cdot_A y}(z,w)&=&P_y(z,x\cdot_A w)+P_x(z,y\cdot_A w),\\
    \label{eq:coh2} P_x(y,z)\cdot_A w&=&x\cdot_A Q(y,z,w)-Q(y,z,x\cdot_A w),\quad\forall~x,y,z,w\in A,
  \end{eqnarray}
 where $Q:\otimes^3\rightarrow
A$ is defined by
$$Q(x,y,z)=[x\cdot_A y,z]_A+[y\cdot_A z,x]_A+[z\cdot_A x,y]_A.$$
Then $(A^*;\ad^*,-\huaL^*)$ is a representation of $A$.}
\end{ex}

\begin{defi}
  A {\bf coherence $F$-manifold algebra} is an $F$-manifold algebra such that \eqref{eq:coh1} and  \eqref{eq:coh2} hold.
\end{defi}

\begin{pro}
Let $(A,\cdot_A,[-,-]_A)$ be an $F$-manifold algebra. Suppose that
there is a nondegenerate symmetric bilinear form $\frak B$ such
that $\frak B$ is invariant in the following sense
\begin{equation}
\frak B(x\cdot_A y,z)=\frak B(x,y\cdot_A z),\;\;\frak
B([x,y]_A,z)=\frak B(x,[y,z]_A),\;\;\forall x,y,z\in A.
\end{equation}
Then $(A,\cdot_A,[-,-]_A)$ is a coherence $F$-manifold algebra.
\end{pro}
\begin{proof}
  By the invariance of $\frak B$, we have
$$\frkB(P_x(y,z),w)=\frkB(z, P_x(y,w))=\frkB(x, Q(y,z,w)),\quad \forall~x,y,z,w\in P.$$
By the above relations, for $x,y,z,w_1,w_2\in P$, we have
\begin{eqnarray*}
\frkB(P_{x\cdot_A y}(z,w_1)-P_y(z,x\cdot_A w_1)-P_x(z,y\cdot_A w_1),w_2)\\
=\frkB(w_1, P_{x\cdot_A y}(z,w_2)-x\cdot_A P_y(z,w_2)-y\cdot_A P_x(z,w_2))
\end{eqnarray*}
and
\begin{eqnarray*}
&&\frkB(P_x(y,z)\cdot_A w_1-x\cdot_A Q(y,z,w_1)+Q(y,z,x\cdot_A w),w_2)\\
&=&\frkB(w_1, -P_{x\cdot_A y}(z,w_2)+x\cdot_A P_y(z,w_2)+y\cdot_A P_x(z,w_2)).
\end{eqnarray*}
By the fact that $A$ is an $F$-manifold algebra and $\frkB$ is nondegenerate, we deduce that \eqref{eq:coh1} and \eqref{eq:coh2} hold. Thus  $(A,\cdot_A,[-,-]_A)$ is a coherence $F$-manifold algebra.
\end{proof}

\subsection{Pre-$F$-manifold algebras}
Recall that a {\bf Zinbiel algebra} is a pair $(A,\diamond)$, where
$A$ is a vector space, and  $\diamond:A\otimes A\longrightarrow A$ is
a bilinear multiplication satisfying that for all $x,y,z\in A$,
\begin{equation}
  x\diamond(y\diamond z)=(y\diamond x)\diamond z+(x\diamond y)\diamond z.
\end{equation}

\begin{lem}\label{lem:den-ass}
Let $(A,\diamond)$ be a Zinbiel algebra. Then $(A,\cdot)$ is a commutative associative algebra, where $x\cdot y=x\diamond y+y\diamond x$. Moreover, for $x\in A$, define $\frkL_{ x}:A\longrightarrow\gl(A)$ by
\begin{equation}\label{eq:dendriform-rep}
\frkL_{ x}y=x\diamond y,\quad\forall~y\in A.
\end{equation}
Then $(A;\frkL)$ is a representation of the commutative associative algebra $(A,\cdot)$.
\end{lem}

Now we are ready to give the main notion in this subsection.

\begin{defi}
  A {\bf  pre-$F$-manifold algebra} is a triple $(A,\diamond,\ast)$, where $(A,\diamond)$ is a Zinbiel algebra and $(A,\ast)$ is a pre-Lie algebra, such that  the following compatibility conditions hold:
  \begin{eqnarray}
    \label{eq:pre-HM 1}F_1(x\cdot y,z,w)=x\diamond F_1(y,z,w)+y\diamond F_1(x,z,w),\\
    \label{eq:pre-HM 2}(F_1(x,y,z)+F_1(x,z,y)+F_2(y,z,x))\diamond w=F_2(y,z,x\diamond w)-x\diamond F_2(y,z,w)
  \end{eqnarray}
 for all $x,y,z,w\in A$. Here $F_1,F_2:\otimes^3 A\longrightarrow A$ are defined by
  \begin{eqnarray}
    F_1(x,y,z)&=&x\ast(y\diamond z)-y\diamond(x\ast z)-[x,y]\diamond z,\\
    F_2(x,y,z)&=&x\diamond(y\ast z)+y\diamond(x\ast z)-(x\cdot y)\ast z
  \end{eqnarray}
and the operation $\cdot$ and bracket $[-,-]$ are defined by
\begin{equation}\label{eq:pHM-operations}
  x\cdot y=x\diamond y+y\diamond x,\quad [x,y]=x\ast y-y\ast x.
\end{equation}
\end{defi}

\begin{rmk}
  If $F_1$ and $F_2$ vanish in the definition of a pre-$F$-manifold algebra $(A,\diamond,\ast)$, then we obtain Aguiar's  notion of a  pre-Poisson algebra. See \cite{A2} for more details.
\end{rmk}

\begin{thm}
  Let $(A,\diamond,\ast)$ be a pre-$F$-manifold algebra. Then
   \begin{itemize}
\item[$\rm(i)$]
   $(A,\cdot,[-,-])$ is an $F$-manifold algebra, where the operation $\cdot$ and bracket $[-,-]$ are given by \eqref{eq:pHM-operations}, which is called the {\bf sub-adjacent
$F$-manifold algebra} of $(A,\diamond,\ast)$  and denoted by $A^c$.
\item[$\rm(ii)$]$(A;L,\frkL)$ is a representation of the sub-adjacent
$F$-manifold algebras $A^c$, where $L$ and $\frkL$ are given by \eqref{eq:defiLpreLie} and \eqref{eq:dendriform-rep}, respectively.
\end{itemize}
\end{thm}
\begin{proof}
(i) By Lemma \ref{lem:pre-Lie-Lie} and Lemma \ref{lem:den-ass}, we deduce that $(A,\cdot)$ is a commutative associative algebra and $(A,[-,-])$ is a Lie algebra.  By direct calculations, we obtain
  \begin{equation}\label{eq:pre-HM 3}
  P_x(y,z)=F_1(x,y,z)+F_1(x,z,y)+F_2(y,z,x),\quad\forall~x,y,z\in A.
  \end{equation}

  By \eqref{eq:pre-HM 1}, \eqref{eq:pre-HM 2} and \eqref{eq:pre-HM 3}, we have
  \begin{eqnarray*}
    &&P_{x\cdot y}(z,w)-x\cdot P_{y}(z,w)-y\cdot P_{x}(z,w)\\
    &=&F_1(x\cdot y,z,w)+F_1(x\cdot y,w,z)+F_2(z,w,x\cdot y)-x\cdot(F_1(y,z,w)+F_1(y,w,z)+F_2(z,w,y)) \\
    &&-y\cdot(F_1(x,z,w)+F_1(x,w,z)+F_2(z,w,x))\\
    &=&\big( F_1(x\cdot y,z,w)-x\diamond F_1(y,z,w)-y\diamond F_1(x,z,w)\big)\\
    &&+\big( F_1(x\cdot y,w,z)-x\diamond F_1(y,w,z)-y\diamond F_1(x,w,z)\big)\\
    &&+\big(F_2(z,w,x\diamond y)-F_1(x,z,w)\diamond y+F_1(x,w,z)\diamond y+F_2(z,w,x)\diamond y+x\diamond F_2(z,w,y)\big)\\
    &&+\big(F_2(z,w,y\diamond x)-F_1(y,z,w)\diamond x+F_1(y,w,z)\diamond x+F_2(z,w,y)\diamond x+y\diamond F_2(z,w,x)\big)\\
    &=& 0.
  \end{eqnarray*}
  Thus $(A,\cdot,[-,-])$ is an $F$-manifold algebra.

(ii) By Lemma \ref{lem:pre-Lie-Lie}, $(A;L)$ is a representation of the sub-adjacent Lie algebra $A^c$. By Lemma \ref{lem:den-ass},  $(A;\frkL)$ is a representation of the commutative associative algebra $(A,\cdot)$. Moreover, note that $F_1(x,y,z)=R_{L,\frkL}(x,y)(z)$ and $F_2(x,y,z)=S_{L,\frkL}(x,y)(z)$. Thus \eqref{eq:pre-HM 1} implies that \eqref{eq:rep 1} holds and, by \eqref{eq:pre-HM 3}, \eqref{eq:pre-HM 2} implies that \eqref{eq:rep 2} holds.  Thus $(A;L,\frkL)$ is a representation of the sub-adjacent $F$-manifold algebra $A^c$.
\end{proof}

\begin{rmk}{\rm
In fact, the operad $\PreFMan$ of pre-$F$-manifold algebras is the
arity splitting of the operad $\FMan$ of $F$-manifold algebras
(into two pieces) in the sense of \cite{PBG} or the disuccessor of
$\FMan$ in the sense of \cite{BBGN}.}

\end{rmk}

\emptycomment{\begin{pro}
 Let $(A,\cdot_A,[-,-]_A)$ be an $F$-manifold algebra. If $\omega\in \wedge^2 A^*$ is nondegenerate, both a  Connes cocycle on the commutative associative $(A,\cdot_A)$ and a symplectic structure on the Lie algebra $(A,[-,-]_A)$, i.e.,
 \begin{eqnarray*}
   \omega(x\cdot_A y,z)+\omega(y\cdot_A z,x)+\omega(z\cdot_A x,y)&=&0,\\
    \omega([x, y]_A,z)+\omega([y, z]_A,x)+\omega([z,x]_A,y)&=&0,\quad\forall~x,y,z\in A.
 \end{eqnarray*}
 Then  $(A,\diamond,\ast)$ is a pre-$F$-manifold algebra, where $\diamond$ and $\ast$ are determined by
  \begin{eqnarray*}
    \omega(x\diamond y,z)=\omega(y,z\cdot_A x),\quad \omega(x\ast y,z)=-\omega(y,[x,z]_A),\quad\forall~x,y,z\in A.
  \end{eqnarray*}
\end{pro}
\begin{proof}
  It follows by a direct calculation. \yh{a few steps}
\end{proof}}

\subsection{Rota-Baxter operators and $\huaO$-operators on $F$-manifold algebras}

The notion of an $\mathcal O$-operator was first given for Lie algebras by Kupershmidt in \cite{K} as a natural generalization of the classical Yang-Baxter equation.

A linear map $T:V\longrightarrow A$ is called an {\bf $\huaO$-operator} on a commutative associative algebra $(A,\cdot_A)$ with respect to a representation  $(V;\mu)$ if $T$ satisfies
  \begin{equation}
    T(u)\cdot_A T(v)=T(\mu(T(u))v+\mu(T(v))u),\quad\forall~u,v\in V.
  \end{equation}
In particular, an $\huaO$-operator on a commutative associative
algebra $(A,\cdot_A)$ with respect to the regular representation
is called a {\bf Rota-Baxter operator of weight zero} or briefly a {\bf Rota-Baxter
operator} on $A$.

\begin{lem}{\rm(\cite{BGN2013A})}
  Let $(A,\cdot_A)$ be a commutative associative algebra and $(V;\mu)$ a representation. Let $T:V\rightarrow A$ be an $\huaO$-operator on  $(A,\cdot_A)$ with respect to $(V;\mu)$. Then there exists a Zinbiel algebra structure on $V$ given by
$$
    u\diamond v=\mu(T(u))v, \quad\forall~u,v\in V.
$$
\end{lem}

A linear map $T:V\longrightarrow \g$ is called an {\bf $\huaO$-operator} on a Lie algebra $(\g,[-,-]_\g)$ with respect to a representation $(V;\rho)$ if $T$ satisfies
\begin{equation}
  [T(u), T(v)]_\g=T\Big(\rho(T(u))(v)-\rho(T(v))(u)\Big),\quad \forall~u,v\in V.
\end{equation}
In particular, an $\huaO$-operator on a Lie algebra
$(\g,[-,-]_\g)$ with respect to the adjoint representation is called a {\bf
Rota-Baxter operator of weight zero} or briefly a {\bf Rota-Baxter operator} on $\g$.

\begin{lem}{\rm(\cite{Bai2007})}
Let $T:V\to \g$ be an $\huaO$-operator  on a Lie algebra $(\g,[-,-]_\g)$ with respect to a representation $(V;\rho)$. Define a multiplication $\ast$ on $V$ by
\begin{equation}
  u\ast v=\rho(T(u))(v),\quad \forall~u,v\in V.
\end{equation}
Then $(V,\ast)$ is a pre-Lie algebra.
 \end{lem}

Let $(V;\rho,\mu)$ be a representation of an $F$-manifold algebra $(A,\cdot_A,[-,-]_A)$.
\begin{defi}
  \begin{itemize}
  \item[{\rm(i)}]A linear operator $T:V\longrightarrow A$ is called  an {\bf $\huaO$-operator on  $A$} if $T$ is both an $\huaO$-operator on the commutative associative algebra $(A,\cdot_A)$ and an $\huaO$-operator on the Lie algebra $(A,[-,-]_A)$;
\item[{\rm(ii)}] A linear operator $\huaB:A\longrightarrow A$ is
called a {\bf
Rota-Baxter operator of weight zero} or briefly a {\bf Rota-Baxter
operator} on $A$, if $\huaB$ is both a Rota-Baxter operator on
the commutative associative algebra $(A,\cdot_A)$ and a
Rota-Baxter operator on the Lie algebra $(A,[-,-]_A)$.
\end{itemize}
\end{defi}

It is obvious that $\huaB: A\longrightarrow A$ is a Rota-Baxter operator on $A$ if and only if $\huaB$ is an  $\huaO$-operator on $A$ with respect to the representation $(A;\ad,\huaL)$.

\begin{ex}\label{ex:id}{\rm
  Let $(A,\diamond,\ast)$ be a pre-$F$-manifold algebra. Then  the identity map $\Id$ is an $\huaO$-operator on $A^c$ with respect  to the representation $(A;L,\frkL)$, where $L$ and $\frkL$ are given by \eqref{eq:defiLpreLie} and \eqref{eq:dendriform-rep} respectively.}
\end{ex}

An $\huaO$-operator on an $F$-manifold algebra gives
a pre-$F$-manifold algebra.

\begin{thm}\label{thm:pre-F-algebra and F-algebra}
Let $(A,\cdot_A,[-,-]_A)$ be an $F$-manifold algebra and $T:V\longrightarrow
A$ an $\huaO$-operator on  $A$ with respect to the representation $(V;\rho,\mu)$. Define new operations $\diamond$ and $\ast$ on $V$ by
$$ u\diamond v=\mu(T(u))v,\quad u\ast v=\rho(T(u))v.$$
Then $(V,\diamond,\ast)$ is a pre-$F$-manifold algebra and $T$ is a homomorphism from $V^c$ to $(A,\cdot_A,[-,-]_A)$.
\end{thm}
\begin{proof}
  First by the fact that $T$ is an $\huaO$-operator on the commutative associative algebra $(A,\cdot_A)$ as well as an $\huaO$-operator on the Lie algebra $(A,[-,-]_A)$ with respect to the representations $(V;\mu)$ and $(V;\rho)$ respectively, we deduce that $(V,\diamond)$ is a Zinbiel algebra and $(V,\ast)$ is a pre-Lie algebra.

Denote by $[u,v]_T:=u\ast v-v\ast u$ and $u\cdot_T v:=u\diamond v+v\diamond u$, we have $T([u,v]_T)=[T(u),T(v)]_A$ and $T(u\cdot_T v)=T(u)\cdot_A T(v)$. By these facts and \eqref{eq:rep 1}, for $v_1,v_2,v_3,v_4\in V$, one has
\begin{eqnarray*}
  &&F_1(v_1\cdot_T v_2,v_3,v_4)-v_1\diamond F_1(v_2,v_3,v_4)-v_2\diamond F_1(v_1,v_3,v_4)\\
  &=&R_{\rho,\mu}(T(v_1)\cdot_A T(v_2),T(v_3))(v_4)-\mu(T(v_1)) R_{\rho,\mu}(T(v_2),T(v_3))(v_4)\\
  &&-\mu(T(v_2)) R_{\rho,\mu}(T(v_1),T(v_3))(v_4)=0,
\end{eqnarray*}
which implies that \eqref{eq:pre-HM 1} holds.

Similarly, by \eqref{eq:rep 2}, one has
\begin{eqnarray*}
  &&(F_1(v_1,v_2,v_3)+F_1(v_1,v_3,v_2)+F_2(v_2,v_3,v_1))\diamond v_4-F_2(v_2,v_3,v_1\diamond v_4)+v_1\diamond F_2(v_2,v_3,v_4)\\
  &=& \mu(P_{T(v_1)}(T(v_2),T(v_3))(v_4)-S_{\rho,\mu}(T(v_2),T(v_3))\mu(T(v_1))(v_4)\\
  &&+\mu(T(v_1)) S_{\rho,\mu}(T(v_2),T(v_3))(v_4)=0,
\end{eqnarray*}
which implies that \eqref{eq:pre-HM 2} holds. Thus, $(V,\diamond,\ast)$ is a  pre-$F$-manifold algebra. It is obvious that $T$ is a homomorphism from $V^c$ to $(A,\cdot_A,[-,-]_A)$.
\end{proof}

\begin{cor}
Let $(A,\cdot_A,[-,-]_A)$ be an $F$-manifold algebra and $T:V\longrightarrow
A$ an $\huaO$-operator on  $A$ with respect to the representation $(V;\rho,\mu)$. Then $T(V)=\{T(v)\mid v\in V\}\subset A$ is a subalgebra of $A$ and there is an induced $F$-manifold algebras structure on $T(V)$ given by
$$T(u)\diamond T(v)=T(u\diamond v),\quad T(u)\ast T(v)=T(u\ast v)$$
for all $u,v\in V$.
\end{cor}

\begin{cor}
 Let $(A,\cdot_A,[-,-]_A)$ be an $F$-manifold algebra and $\huaB:A\longrightarrow
A$ a Rota-Baxter operator. Define new operations on $A$ by
$$x\diamond y=\huaB(x)\cdot_A y,\quad x\ast y=[\huaB(x),y]_A.$$
Then $(A,\diamond,\ast)$ is a pre-$F$-manifold algebra and $\huaB$ is a homomorphism from the sub-adjacent $F$-manifold algebras $(A,\cdot_\huaB,[-,-]_\huaB)$ to $(A,\cdot_A,[-,-]_A)$, where $x\cdot_\huaB y=x\diamond y+y\diamond x$ and $[x,y]_\huaB=x\ast y-y\ast x$.

\end{cor}

\begin{cor}
 Consider the $F$-manifold algebra $(A,\cdot,[-,-])$ given by Proposition \ref{ex:derivation-HMA}. If $\huaB:A\longrightarrow A$ is a Rota-Baxter operator on the commutative associative algebra $(A,\cdot)$ and satisfies $\huaB\circ D =D\circ B$, then $\huaB:A\longrightarrow A$ is a Rota-Baxter operator on the Lie algebra $(A,[-,-])$.
 Thus $(A,\diamond_\huaB,\ast_\huaB)$ is a pre-$F$-manifold algebra, where
 $$x\diamond_\huaB y=\huaB(x)\cdot y,\quad x\ast_\huaB y=\huaB(x)\cdot D(y)-y\cdot D(\huaB(x))$$
 and $(A,\cdot_\huaB,[-,-]_\huaB)$ is the sub-adjacent $F$-manifold algebra with
 $$x\cdot_\huaB y=x\diamond_\huaB y+y\diamond_\huaB x,\quad [x,y]_\huaB=x\ast_\huaB y-y\ast_\huaB x.$$
\end{cor}

\begin{ex}{\rm
  We put $A=C^1([0,1])$. Then $(A,\cdot,[-,-])$ with the following multiplication and bracket
  \begin{eqnarray*}
  f\cdot g&=& fg,\\
{[f,g]}&=&fg'-gf',\quad\forall~f,g\in A
  \end{eqnarray*}
  is an $F$-manifold algebra. It is well-known that the integral operator is a Rota-Baxter operator:
$$
   \huaB:A\rightarrow A,\quad \huaB(f)(x):=\int^x_0f(t)dt.
$$
It is easy to see that
$$\huaB\circ \partial_x=\partial_x\circ \huaB=\Id.$$
Thus $(A,\diamond_\huaB,\ast_\huaB)$ is a pre-$F$-manifold algebra, where
 $$f\diamond_\huaB g=g\int^x_0f(t)dt,\quad f\ast_\huaB g=g'\int^x_0f(t)dt-f\cdot g$$
 and $(A,\cdot_\huaB,[-,-]_\huaB)$ is the sub-adjacent $F$-manifold algebra with
 $$f\cdot_\huaB g=f\int^x_0g(t)dt+g\int^x_0f(t)dt,\quad [f,g]_\huaB=f'\int^x_0g(t)dt -g'\int^x_0f(t)dt.$$

}
\end{ex}

\begin{ex}{\rm
Consider the $F$-manifold algebra $(A,\cdot,[-,-])$ given by Example \ref{ex:3-dimensional F-algebra}. It is straightforward to check that $\huaB$ given by
\begin{eqnarray*}
  \huaB(e_1)=re_1,\quad \huaB(e_2)=2se_1+\frac{3}{2}re_2,\quad
  \huaB(e_3)=te_1+3se_2+3re_3,\quad r,s,t\in \K
\end{eqnarray*}
is a Rota-Baxter operator on the $F$-manifold algebra $A$. Thus $(A,\diamond_\huaB,\ast_\huaB)$ is a pre-$F$-manifold algebra, where
 \begin{eqnarray*}
 e_2\diamond_\huaB e_3&=&\frac{3}{2}re_1,\quad e_3\diamond_\huaB e_2=3re_1,\quad e_3\diamond_\huaB e_3=3se_1+3re_2;\\
 e_2\ast_\huaB e_3&=&-\frac{3}{2}are_1,\quad e_3\ast_\huaB e_2=3are_1,\quad e_3\ast_\huaB e_3=-3ase_1
 \end{eqnarray*}
 and $(A,\cdot_\huaB,[-,-]_\huaB)$ is the sub-adjacent $F$-manifold algebra with
\begin{eqnarray*}
e_2\cdot_\huaB e_3&=&\frac{9}{2}re_1,\quad e_3\cdot_\huaB e_3=6se_1+6re_2;\\
{[e_2,e_3]_\huaB}&=&-\frac{9}{2}are_1.
\end{eqnarray*}

}
\end{ex}

At the end of this section, we give a
necessary and sufficient condition on an $F$-manifold algebra admitting a
pre-$F$-manifold algebra structure.

\begin{pro}\label{pro:nsc}
 Let $(A,\cdot_A,[-,-]_A)$ be an $F$-manifold algebra. There is a pre-$F$-manifold algebra structure on $A$ such that its sub-adjacent $F$-manifold algebra is exactly $(A,\cdot_A,[-,-]_A)$ if and only if there exists an invertible $\huaO$-operator on $(A,\cdot_A,[-,-]_A)$.
\end{pro}

\begin{proof} If $T:V\longrightarrow
A$ is an invertible $\huaO$-operator on  $A$ with respect to the representation $(V;\rho,\mu)$, then the compatible pre-$F$-manifold algebra structure on $A$ is given by
$$x\diamond y=T(\mu(x)(T^{-1}(y))),\quad x\ast y=T(\rho(x)(T^{-1}(y)))$$
for all $x,y\in P$.

Conversely, let $(A,\diamond,\ast)$ be a pre-$F$-manifold algebra and $(A,\cdot_A,[-,-]_A)$ the sub-adjacent $F$-manifold algebra. Then the identity map $\Id$ is an $\huaO$-operator on $A$ with respect to the representation $(A;L,\frkL)$.
\end{proof}

\begin{cor}
Let $(A,\cdot_A,[-,-]_A)$ be a  coherence $F$-manifold algebra. Let $\omega\in\wedge^2A^*$ be a cyclic $2$-cocycle in the sense of Connes on the commutative associative algebra $(A,\cdot_A)$, i.e.
$$
 \omega(x\cdot_A y,z)+\omega(y\cdot_A z,x)+\omega(z\cdot_A x,y)=0,\quad\forall~x,y,z\in A,
$$
as well as a symplectic structure on the Lie algebra  $(A,[-,-]_A)$, i.e.
$$\omega([x,y]_A,z)+ \omega([y,z]_A,x)+ \omega([z,x]_A,y)=0,\quad \forall x,y,z\in  A.$$
Then there is a compatible pre-$F$-manifold algebra structure on $A$ given by
$$
\omega(x\diamond y,z)=\omega(y,x\cdot_A z),\quad \omega(x\ast y,z)=\omega(y,[z,x]_A).
$$
\end{cor}

 \begin{proof}
   Since $(A,\cdot_A,[-,-]_A)$ is a  coherence $F$-manifold algebra, $(A^*;\ad^*,-\huaL^*)$ is a representation of $A$. By the fact that $\omega$ is a cyclic $2$-cocycle, $(\omega^\sharp)^{-1}$ is an $\huaO$-operator on the commutative associative algebra $(A,\cdot_A)$ with respect to the representation $(A^*,-\huaL^*)$, where $\omega^\sharp:A\longrightarrow A^*$ is defined by $\langle\omega^\sharp(x),y\rangle=\omega(x,y)$. By the fact that $\omega$ is a symplectic structure, $(\omega^\sharp)^{-1}$ is an $\huaO$-operator on the Lie algebra $(A,[\cdot,\cdot]_A)$ with respect to the representation $(A^*,\ad^*)$. Thus, $(\omega^\sharp)^{-1}$ is an $\huaO$-operator on the  coherence $F$-manifold algebra $(A,\cdot_A,[-,-]_A)$ with respect to the representation $(A^*;\ad^*,-\huaL^*)$. By Proposition \ref{pro:nsc}, there is a compatible pre-$F$-manifold algebra structure on $A$ given as above.
 \end{proof}

\section{Dual pre-$F$-manifold algebras and average operators on $F$-manifold algebras}\label{sec:dual pre-F-algebras}
It is well known that the Koszul dual of the operad $\PreLie$ of pre-Lie algebras is the operad $\Perm$ of permutative algebras
(\cite{Lod2}). Recall that a {\bf permutative algebra} is a pair
$(A,\bullet)$, where $A$ is a vector space and $\bullet:A\otimes
A\longrightarrow A$ is a bilinear multiplication satisfying that
for all $x,y,z\in A$,
\begin{equation}
  x\bullet(y\bullet z)=(x\bullet y)\bullet z=(y\bullet x)\bullet z.
\end{equation}

It is well known that the Koszul dual of the operad $\Zinb$ of Zinbiel algebras is the operad $\Leib$ of Leibniz algebras (\cite{ChaLiv}).
Recall that a {\bf Leibniz algebra} is a pair $(\g,\{-,-\})$,
where $\g$ is a vector space and $\{-,-\}:\g\otimes
\g\longrightarrow \g$ is a bilinear bracket satisfying that for
all $x,y,z\in \g$,
\begin{equation}
 \{x,\{y,z\}\}=\{\{x,y\},z\}+\{y,\{x,z\}\}.
\end{equation}

Now we give the definition of a dual pre-$F$-manifold algebra.

\begin{defi}
A {\bf dual pre-$F$-manifold algebra} is a triple $(A,\bullet,\{-,-\})$, where $(A,\bullet)$ is a permutative algebra and $(A,\{-,-\})$ is a Leibniz algebra, such that for all $x,y,z,w\in A$, the following conditions hold:
\begin{eqnarray}
\label{eq:dule relation1}G_1(x\bullet y,z,w)&=&x\bullet G_1(y,z,w)+y\bullet G_1(x,z,w),\\
\label{eq:dule relation2}G_2(y\bullet x,z,w)&=& y\bullet G_2(x,z,w)-G_1(y,z,w)\bullet x,\\
\label{eq:dule relation3}\{x,y\}\bullet z+\{y,x\} \bullet z&=& 0,
\end{eqnarray}
where $G_1,G_2:\otimes^3 A\longrightarrow A$ are defined by
  \begin{eqnarray}
   G_1(x,y,z)&=&\{x,y\bullet z\}-\{x,y\}\bullet z-y\bullet \{x,z\} ,\\
   G_2(x,y,z)&=&\{y\bullet z,x\}-z\bullet \{y,x\}-y\bullet\{z,x\}.
  \end{eqnarray}
\end{defi}

Recall that an {\bf average operator} on a commutative associative algebra $(A,\cdot)$ is a linear map $\alpha:A\rightarrow A$ such that
$$\alpha(x)\cdot \alpha(y)=\alpha(\alpha(x)\cdot y),\quad\forall~x,y\in A.$$
\begin{lem}\label{lem:permutative algebra}{\rm(\cite{A2})}
  Let $(A,\cdot)$  be a  commutative associative algebra and $\alpha:A\rightarrow A$ an average operator. Define a multiplication $\bullet$ on $A$ by
\begin{equation}
 x\bullet y= \alpha(x)\cdot y,\quad \forall~x,y\in A.
\end{equation}
Then $(A,\bullet)$ is a permutative algebra.
\end{lem}

Recall that an {\bf average operator} on a Lie algebra $(\g,[-,-])$ is a linear map $\alpha:\g\rightarrow \g$ such that
$$[\alpha(x), \alpha(y)]=\alpha([\alpha(x), y]),\quad\forall~x,y\in \g.$$
\begin{lem}\label{lem:Leibniz algebra}{\rm(\cite{A2})}
  Let  $(\g,[-,-])$ be a  Lie algebra and $\alpha:\g\rightarrow \g$ an average operator. Define a bracket $\{-,-\}$ on $\g$ by
\begin{equation}
 \{x, y\}= [\alpha(x), y],\quad \forall~x,y\in \g.
\end{equation}
Then $(\g,\{-,-\})$ is a Leibniz algebra.
\end{lem}

\begin{defi}
  Let $(A,\cdot_A,[-,-]_A)$ be an $F$-manifold algebra. A linear operator $\alpha:A\longrightarrow A$ is called  an {\bf average operator on  $A$} if $\alpha$ is both an average operator on the commutative associative algebra $(A,\cdot_A)$ and an average operator on the Lie algebra $(A,[-,-]_A)$.
\end{defi}

\begin{thm}
 Let $(A,\cdot_A,[-,-]_A)$ be an $F$-manifold algebra and $\alpha:A\rightarrow A$ an average operator. Define new operations on $A$ by
\begin{equation}
x\bullet y= \alpha(x)\cdot_A y\quad \mbox{  and  }\quad \{x, y\}= [\alpha(x), y]_A,\quad \forall~x,y\in A.
\end{equation}
Then $(A,\bullet,\{-,-\})$ is a dual pre-$F$-manifold algebras.
\end{thm}
\begin{proof}
  Let $(A,\cdot_A,[-,-]_A)$ be an $F$-manifold algebra. Since $\alpha:A\rightarrow A$ is an average operator, we have $[\alpha(x), \alpha(y)]_A=\alpha([\alpha(x), y]_A)$ and $\alpha(x)\cdot_A \alpha(y)=\alpha(\alpha(x)\cdot_A y)$. By Lemma \ref{lem:permutative algebra} and Lemma \ref{lem:Leibniz algebra}, $(A,\cdot_A)$ is a permutative algebra and $(A,[-,-]_A)$ is a Leibniz algebra. By a direct calculation, we have
  \begin{eqnarray}\label{eq:dule relation4}
  P_{\alpha(x)}(\alpha(y),z)=P_{\alpha(x)}(z,\alpha(y))=G_1(x,y,z),\quad  P_x(\alpha(y),\alpha(z))=-G_2(x,y,z).
  \end{eqnarray}
By the definition of an average operator and \eqref{eq:dule relation4}, we have
\begin{eqnarray*}
  &&P_{\alpha(x)\cdot_A \alpha(y)}(\alpha(z),w)-\alpha(x)\cdot_A P_{\alpha(y)}(\alpha(z),w)-\alpha(y)\cdot_A P_{\alpha(x)}(\alpha(z),w)\\
 &=& P_{\alpha(x\cdot_A y)}(\alpha(z),w)-x\bullet P_{\alpha(y)}(\alpha(z),w)-y\bullet P_{\alpha(x)}(\alpha(z),w)\\
 &=&G_1(x\bullet y,z,w)-x\bullet G_1(y,z,w)-y\bullet G_1(x,z,w),
\end{eqnarray*}
which implies that \eqref{eq:dule relation1} holds.

Similarly, we have
\begin{eqnarray*}
  &&P_{x\cdot_A \alpha(y)}(\alpha(z),\alpha(w))-x\cdot_A P_{\alpha(y)}(\alpha(z),\alpha(w))-\alpha(y)\cdot_A P_{x}(\alpha(z),\alpha(w))\\
 &=& P_{y\bullet x}(\alpha(z),w)-x\cdot_A \alpha(G_1(y,z,w))-y\bullet P_{x}(\alpha(z),\alpha(w))\\
 &=&-G_2(y\bullet x,z,w)-G_1(y,z,w)\bullet x+y\bullet G_2(x,z,w),
\end{eqnarray*}
which implies that \eqref{eq:dule relation2} holds.

Since $\alpha([\alpha(x),y]_A)=\alpha([x,\alpha(y)]_A)=[\alpha(x),\alpha(y)]_A)$, we have
\begin{eqnarray*}
\{x,y\}\bullet z+\{y,x\}\bullet z&=&\alpha ([\alpha(x),y]_A)\cdot_A z+ ([\alpha(y),x]_A)\cdot_A z\\
&=&([\alpha(x),\alpha(y)]_A+ [\alpha(y),\alpha(x)]_A)\cdot_A z=0,
\end{eqnarray*}
which implies that \eqref{eq:dule relation3} holds.
\end{proof}

\begin{pro}
 Consider the $F$-manifold algebra $(A,\cdot,[-,-])$ given by Proposition \ref{ex:derivation-HMA}. If $\alpha:A\longrightarrow A$ is an average operator on the commutative associative algebra $(A,\cdot)$ and satisfies $\alpha\circ D =D\circ \alpha$, then $\alpha:A\longrightarrow A$ is an average operator on the Lie algebra $(A,[-,-])$.
 Thus $(A,\bullet,\{-,-\})$ is a dual pre-$F$-manifold algebra, where
 $$x\bullet y= \alpha(x)\cdot y\quad \mbox{  and  }\quad \{x, y\}= \alpha(x)\cdot D(y)-y\cdot D(\alpha(x)),\quad \forall~x,y\in A.$$
\end{pro}
\begin{proof}
  It follows by a direct calculation.
\end{proof}

\begin{ex}{\rm
Consider the $F$-manifold algebra $(A,\cdot,[-,-])$ given by Example \ref{ex:3-dimensional F-algebra}. It is straightforward to check that $\alpha$ given by
\begin{eqnarray*}
 \alpha(e_1)=re_1,\quad \alpha(e_2)=se_1+re_2,\quad
  \alpha(e_3)=te_1+se_2+re_3,\quad r,s,t\in \K
\end{eqnarray*}
is an average operator on the $F$-manifold algebra $A$. Thus $(A,\bullet,\{-,-\})$ is a dual pre-$F$-manifold algebra, where
 \begin{eqnarray*}
 e_2\bullet e_3&=&re_1,\quad e_3\bullet e_2=re_1,\quad e_3\bullet e_3=se_1+re_2;\\
 \{e_2,e_3\}&=&-are_1,\quad \{e_3, e_2\}=are_1,\quad \{e_3, e_3\}=-ase_1\quad a,r,s\in \K.
 \end{eqnarray*}

}
\end{ex}


\begin{thebibliography}{abc}
\bibitem{A2}  M. Aguiar, Pre-Poisson algebras. \emph{Lett. Math. Phys.} 54 (2000), 263-277.

\bibitem{Bai2007}
C. Bai, A unified algebraic approach to classical Yang-Baxter equation. \emph{J. Phys. A: Math. Theo.} 40 (2007), 11073-11082.

\bibitem{BGN2013A}
C. Bai, L. Guo and  X. Ni,  $\huaO$-operators on associative algebras and associative Yang-Baxter equations. \emph{Pacific J. Math.} 256 (2012), 257-289.

\bibitem{BBGN} C. Bai, O. Bellier, L. Guo and X. Ni, Spliting of operations, Manin products and Rota-Baxter operators. \emph{Int. Math. Res. Not.}  3 (2013), 485-524.

\bibitem{Bakalov}
B. Bakalov and V. Kac, Field algebras. \emph{Int. Math. Res. Not.} 3 (2003), 123-159.

\bibitem{Ban}
R. Bandiera, Formality of Kapranov¡¯s brackets in K\"{a}hler geometry via pre-Lie deformation theory. \emph{Int. Math. Res. Not.} 21 (2016), 6626-6655.

\bibitem{Burde} D. Burde, Left-symmetric algebras and pre-Lie algebras in geometry and physics. \emph{ Cent. Eur. J. Math.} 4 (2006), 323-357.

\bibitem{ChaLiv}
 F. Chapoton and M. Livernet, Pre-Lie algebras and the rooted trees operad. \emph{ Int. Math. Res. Not.} 8 (2001), 395-408.

\bibitem{CK}
A. Connes and D. Kreimer, Renormalization in quantum field theory and the Riemann-Hilbert problem. I. The Hopf algebra structure of graphs and the main theorem. {\em Comm. Math. Phys.}   210  (2000), 249-273.


\bibitem{DS04}
L. David and I. A. B. Strachan, Compatible metrics on a manifold and nonlocal bi-Hamiltonian structures. \emph{Int. Math. Res. Not.} 66 (2004), 3533-3557.

\bibitem{DS11}
L. David, and I. A. B. Strachan, Dubrovin¡¯s duality for $F$-manifolds with eventual identities. \emph{Adv. Math.} 226 (2011), 4031-4060.

\bibitem{DSV}
V. Dotsenko, S. Shadrin and B. Vallette, Pre-Lie deformation theory. \emph{ Mosc. Math. J.} 16 (2016), 505-543.

\bibitem{Dot}
V. Dotsenko, Algebraic structures of F-manifolds via pre-Lie algebras. \emph{Ann. Mat. Pura Appl.} 198 (2019), 517-527.

\bibitem{Dub95}
B. Dubrovin, Geometry of 2D topological field theories. \emph{Lecture Notes in Math}, 1620 (1995).



\bibitem{cohomology of pre-Lie}
 A. Dzhumadil$\rm{'}$daev, Cohomologies and deformations of right-symmetric
algebras. \emph{J. Math. Sci.} 93 (1999), 836-876.

 \bibitem{EGK}
 K. Ebrahimi-Fard, L. Guo and D. Kreimer, Integrable renormalization II: the general case. \emph{Ann. Henri Poincare} 6 (2005), 369-395.

\bibitem{Fil}
V. T. Filippov, A class of simple nonassociative algebras, \emph{Mat. Zametki} 45 (1989), 101-105.

\bibitem{Foissy}
L. Foissy, The Hopf algebra of Fliess operators and its dual pre-Lie algebra. \emph{Commun. Algebra} 43 (2015), 4528-4552.

\bibitem{GD}
I. M. Gel¡¯fand and I. Ya. Dorfman, Hamiltonian operators and algebraic structures related to
them. \emph{Funct. Anal. Appl.} 13 (1979) 248-262.


\bibitem{Gub}
L. Guo,  An introduction to Rota-Baxter algebra. Surveys of Modern Mathematics, 4. International Press, Somerville, MA; Higher Education Press, Beijing, 2012. xii+226 pp.

\bibitem{GK}
L. Guo and W. Keigher, Baxter algebras and shuffle products. \emph{Adv. Math.} 150 (2000), 117-149.


 \bibitem{Her02}
C. Hertling, Frobenius Manifolds and Moduli Spaces for Singularities. \emph{ Cambridge Tracts in Math.} Cambridge
University Press, 2002.

\bibitem{HerMa}
 C. Hertling and Y. I. Manin, Weak Frobenius manifolds. \emph{Int. Math. Res. Not.} 6 (1999), 277-286.

\bibitem{K}
B. A. Kupershmidt, What a classical $r$-matrix really is, \emph {J.
Nonlinear Math. Phys.} 6 (1999), 448-488.



\bibitem{LYP}
Y. P. Lee, Quantum K-theory, I: Foundations. \emph{Duke Math. J.} 121 (3) (2004), 389-424.

\bibitem{Lichnerowicz}
A. Lichnerowicz  and A. Medina, On Lie groups with left-invariant symplectic or K$\rm\ddot{a}$hlerian structures. \emph{Lett.
Math. Phys.} 16 (1988), 225-235.

\bibitem{Liv}
M. Livernet, Rational homotopy of Leibniz algebras. \emph{Manuscripta Math.} 96 (1998), 295-315.

\bibitem{Lod1}
J.-L. Loday, Cup product for Leibniz cohomology and dual Leibniz algebras. \emph{Math. Scand.} 77, Univ. Louis Pasteur, Strasbourg, 1995, pp. 189-196.

 \bibitem{Lod2}
J.-L. Loday, Dialgebras. In Dialgebras and related operads. \emph{Lecture Notes in Math.} 1763, Springer, Berlin 2001, 7-66.

 \bibitem{LPR11}
P. Lorenzoni, M. Pedroni and A. Raimondo, $F$-manifolds and integrable systems of hydrodynamic type. \emph{Arch. Math.} (Brno) 47 (2011), 163-180.

\bibitem{MT}S. Majid and W. Tao, Noncommutative differentials on Poisson-Lie groups and pre-Lie algebras.  \emph{Pacific J. Math.} 284 (2016), 213-256.

\bibitem{Mansuy}
A. Mansuy, Preordered forests, packed words and contraction algebras. \emph{J. Algebra} 411 (2014), 259-311.



\bibitem{Merku}
S. A. Merkulov, Operads, deformation theory and $F$-manifolds. \emph{Asp. Math.} 36 (2004), 213-251.

\bibitem{PBG} J. Pei, C. Bai and L. Guo, Splitting of operads and Rota-Baxter operators on operads. {\em Appl. Cate. Stru.}  25 (2017), 505-538.

\bibitem{Xu}
X. Xu, On simple Novikov algebras and their irreducible modules. \emph{J. Algebra} 185 (1996), 905-934.


\end{thebibliography}
 \end{document}